\documentclass[11pt]{article}
\usepackage[margin=1.05in]{geometry}
\usepackage{amsmath,amssymb,amsthm,mathtools}
\usepackage{booktabs}
\usepackage{microtype}
\usepackage[hidelinks]{hyperref}

\numberwithin{equation}{section}

\newtheorem{theorem}{Theorem}[section]
\newtheorem{proposition}[theorem]{Proposition}
\newtheorem{lemma}[theorem]{Lemma}
\newtheorem{corollary}[theorem]{Corollary}
\theoremstyle{remark}
\newtheorem{remark}[theorem]{Remark}

\newcommand{\F}[2]{{}_{#1}F_{#2}}
\newcommand{\Ffourthree}{{}_4F_3}
\newcommand{\dd}{\,\mathrm d}
\newcommand{\eps}{\varepsilon}

\newcommand{\cC}{\mathcal C}
\newcommand{\cL}{\mathcal L}

\title{Coxeter Symmetry in Hypergeometric Functions and Elliptic Integral Moments}
\author{Dianbin Bao\\[1mm]
\small Department of Mathematics, Alvernia University, Reading, Pennsylvania 19607, USA\\
\small \texttt{dianbin.bao@alvernia.edu}}
\date{}

\begin{document}
\maketitle

\begin{abstract}
We develop a hypergeometric framework for elliptic integral moments organized
around Mishev's completed Saalsch\"utzian $L$-function and its $W(D_5)$
symmetry.  The Bailey--Mishev very-well-poised ${}_7F_6(1)$ representation
places complementary elliptic integral moments and harmonic hypergeometric
sums in a common parameter space.  We first give collision-based
hypergeometric derivations of the three complementary-modulus moment families
involving $K'^2$, $E'K'$, and $E'^2$.  The $E'^2$ identity is also recovered
independently from a Barnes-contiguous relation mirroring the differential
system of the elliptic integrals.  We then derive a hypergeometric evaluation
of the harmonic-series representation of the four-Bessel moment $s_{4,0}$.
At symmetric parameter points, Coxeter symmetry forces selected Taylor
coefficients to vanish along symmetry-adapted paths, producing families of
harmonic and special-value identities involving values of the Riemann zeta
function and the same Bessel period.  The resulting picture shows that these
elliptic-moment, harmonic-sum, and special-value identities arise from a
common $W(D_5)$ symmetry of generalized hypergeometric functions.
\end{abstract}


\section{Introduction}
The hypergeometric background of the present work has a classical origin.
Very-well-poised ${}_7F_6(1)$ transformations go back to Bailey
\cite{MR185155}.  Coxeter-group structures underlying transformations of
Saalsch\"utzian ${}_4F_3(1)$ series were developed by
Formichella, Green, and Stade \cite{MR2765604}.  Mishev subsequently
introduced a gamma-completed function $L$ which can be represented both as
a combination of two Saalsch\"utzian ${}_4F_3(1)$ series and as a
very-well-poised ${}_7F_6(1)$ series, and showed that its invariance group
is isomorphic to $W(D_5)$ \cite{MR2834915}.  This completed function is the
global hypergeometric object used throughout the present paper.

A second line of development concerns moments of complete elliptic
integrals, including complementary-modulus moments.  Wan obtained
systematic evaluations involving elliptic integrals of the first and
second kinds and their complementary counterparts, with many of the
resulting formulas taking very-well-poised hypergeometric form
\cite{MR2845511}.  Their arithmetic significance was developed further in
work relating elliptic-integral moments to critical modular $L$-values and
lattice sums \cite{MR3338042,MR3490555}.  McCrorie later studied related
special values and contiguous relations directly in the language of
very-well-poised ${}_7F_6(1)$ series, making explicit several
hypergeometric consequences of Wan's formulas \cite{MR4332182}.  These
results provide important special instances of the hypergeometric
structures considered here.

The present viewpoint differs in that the individual ${}_7F_6(1)$
representations are not taken as isolated starting points.  Instead, we
recover the complementary $K'^2$, $E'K'$, and $E'^2$ moment families from
Mishev's completed $L$-function and place them simultaneously in a common
$W(D_5)$-symmetric parameter space.  In this setting the hypergeometric
representations, reflection identities, collision limits, and contiguous
relations arise from the same global object.  In particular, previously
occurring very-well-poised specializations appear as individual points or
loci inside a larger symmetry framework.

A parallel source of identities comes from Bessel moments.  Bailey,
Borwein, Broadhurst, and Glasser studied elliptic-integral evaluations of
Bessel moments arising in mathematical physics and obtained, among many
other formulas, the harmonic-series representation associated with the
four-Bessel moment $s_{4,0}$ \cite{MR2450513}; see also
\cite{MR3573666} for later developments connecting Bessel moments,
modular forms, and $L$-series.  We show that the harmonic factor in the
BBBG series can itself be generated by differentiation in Mishev's
parameter space.  A $W(D_5)$ reflection and a symmetric collision then
give a purely hypergeometric route to its evaluation.

The Bessel calculation then suggests a local use of the group action.  At a
symmetric point fixed by a finite-order Weyl element, an eigenvector in the
tangent representation determines a symmetry-adapted parameter path.
Invariance forces Taylor coefficients in incompatible degrees to vanish.
Because parameter differentiation of Pochhammer symbols produces harmonic
numbers and polygamma values, these selection rules become arithmetic
identities.  We develop this mechanism at the symmetric Bessel point and
record its higher-order consequences.

Accordingly, the paper has two complementary levels.  Globally, $W(D_5)$
reflections, collisions, and contiguous shifts organize the elliptic and
Bessel hypergeometric formulas.  Locally, finite-order Coxeter symmetry at a
symmetric point constrains parameter derivatives and produces harmonic
identities.  The common organizing object throughout is Mishev's completed
$L$-function.

The paper is organized as follows.
Section~\ref{sec:mishev} recalls Mishev's completed $L$-function,
its $W(D_5)$ invariance, and the Bailey--Mishev and Barnes
representations used throughout.
Section~\ref{sec:complementary-moments} develops the three complementary
elliptic-moment families and their contiguous-square structure.
Section~\ref{sec:bbbg} treats the BBBG four-Bessel moment.
Section~\ref{sec:eigenpath} develops the Coxeter eigenpath principle,
the $W(D_5)$-fixed deformation of the BBBG point, and the first
arithmetic consequences.
Section~\ref{sec:higher-order} derives the higher-order harmonic
identities.
Finally, Section~\ref{sec:conclusion} summarizes the global and local
roles of $W(D_5)$ and the structural relations among the preceding
constructions.

\section{Mishev's completed $L$-function and $W(D_5)$ symmetry}
\label{sec:mishev}

We first fix the hypergeometric notation used throughout the paper.  For
complex parameters $a_1,\ldots,a_p$ and $b_1,\ldots,b_q$, with no lower
parameter a nonpositive integer, write
\begin{equation}
	{}_pF_q\!\left(
	\begin{matrix}
		a_1,\ldots,a_p\\
		b_1,\ldots,b_q
	\end{matrix};z
	\right)
	:=
	\sum_{n=0}^{\infty}
	\frac{(a_1)_n\cdots(a_p)_n}
	{(b_1)_n\cdots(b_q)_n}
	\frac{z^n}{n!},
\end{equation}
where $(a)_0:=1$ and
\begin{equation}
	(a)_n:=a(a+1)\cdots(a+n-1)
	=\frac{\Gamma(a+n)}{\Gamma(a)},
	\qquad n\ge1,
\end{equation}
whenever the gamma quotient is defined. We use these series in their domains of convergence and, when necessary, their analytic or meromorphic continuations in the parameters.

Following Mishev \cite{MR2834915}, write
\begin{equation}
	p=(a,b,c,d;e;f,g)
\end{equation}
for a point on the Saalsch\"utzian hyperplane
\begin{equation}
	e+f+g-a-b-c-d=1.
	\label{eq:saalschutz}
\end{equation}
For generic parameters satisfying \eqref{eq:saalschutz}, define the
gamma-completed function
\begin{align}
	L(a,b,c,d;e;f,g)
	={}&
	\frac{{}_4F_3\!\left(
		\begin{matrix}
			a,b,c,d\\
			e,f,g
		\end{matrix};1\right)}
	{\sin(\pi e)\Gamma(e)\Gamma(f)\Gamma(g)
		\prod_{\xi\in\{a,b,c,d\}}\Gamma(1+\xi-e)}
	\notag\\
	&-
	\frac{{}_4F_3\!\left(
		\begin{matrix}
			1+a-e,1+b-e,1+c-e,1+d-e\\
			1+f-e,1+g-e,2-e
		\end{matrix};1\right)}
	{\sin(\pi e)\Gamma(a)\Gamma(b)\Gamma(c)\Gamma(d)
		\Gamma(1+f-e)\Gamma(1+g-e)\Gamma(2-e)}.
	\label{eq:mishev-L}
\end{align}
We use the meromorphic continuation of \eqref{eq:mishev-L} at parameter
values for which the individual terms in this expression are singular.

A fundamental symmetry is the affine involution
\begin{equation}
	\mathcal A(a,b,c,d;e;f,g)
	=
	\left(
	a,b,g-c,g-d;
	1+a+b-f;
	1+a+b-e,g
	\right).
	\label{eq:mishev-involution}
\end{equation}
Mishev's two-term relation is
\begin{equation}
	L(p)=L(\mathcal A p).
	\label{eq:mishev-reflection}
	\end{equation}
Together with permutations of $a,b,c,d$ and the interchange
$f\leftrightarrow g$, these transformations generate an invariance group
isomorphic to the Weyl group $W(D_5)$
\cite{MR2765604,MR2834915}.  This group action is the global symmetry
underlying the reflection and collision arguments used below.

Set
\begin{equation}
	\alpha=d+g-e.
\end{equation}
The Bailey--Mishev representation expresses the same completed function
as
\begin{align}
	L(a,b,c,d;e;f,g)
	={}&
	\frac{\Gamma(1+d+g-e)}
	{\pi\Gamma(g)\Gamma(1+g-e)\Gamma(f-d)
		\Gamma(1+a+d-e)\Gamma(1+b+d-e)\Gamma(1+c+d-e)}
	\notag\\
	&\times
	{}_7F_6\!\left(
	\begin{matrix}
		\alpha,1+\frac{\alpha}{2},
		g-a,g-b,g-c,d,1+d-e\\
		\frac{\alpha}{2},
		1+a+d-e,1+b+d-e,1+c+d-e,1+g-e,g
	\end{matrix};1
	\right).
	\label{eq:bailey-mishev}
\end{align}
The ${}_7F_6(1)$ in \eqref{eq:bailey-mishev} is very well-poised in the
classical sense.  We shall use only the explicit Bailey--Mishev form above,
rather than the general theory of very-well-poised series.  This
representation lies in the classical hypergeometric tradition going back
to Bailey \cite{MR185155}; its role here is that it converts the
$W(D_5)$ action on the completed $L$-function into explicit transformations
of ${}_7F_6(1)$ series.

For the contiguous arguments we shall also use Mishev's Barnes
representation
\begin{align}
	L(a,b,c,d;e;f,g)
	={}&
	\frac{1}
	{\pi\Gamma(a)\Gamma(b)\Gamma(c)\Gamma(d)
		\Gamma(1+a-e)\Gamma(1+b-e)
		\Gamma(1+c-e)\Gamma(1+d-e)}
	\notag\\
	&\times
	\frac{1}{2\pi i}
	\int_{\mathfrak C}
	\frac{
		\Gamma(a+t)\Gamma(b+t)\Gamma(c+t)\Gamma(d+t)
		\Gamma(1-e-t)\Gamma(-t)}
	{\Gamma(f+t)\Gamma(g+t)}
	\,\dd t ,
	\label{eq:mishev-barnes}
\end{align}
where $\mathfrak C$ is a Barnes contour separating the appropriate left and
right sequences of poles \cite{MR2834915}.

The two representations will be used for different purposes.
The Bailey--Mishev formula \eqref{eq:bailey-mishev} is particularly
effective for Weyl-group reflections, parameter cancellations, and the
special ${}_7F_6(1)$ evaluations arising from elliptic and Bessel moments,
whereas the Barnes representation \eqref{eq:mishev-barnes} is naturally
adapted to contiguous shifts and telescoping arguments.

\section{Complementary elliptic-integral moments}
\label{sec:complementary-moments}

We use the standard complete elliptic integrals
\begin{equation}
	K(k):=\int_0^{\pi/2}\frac{d\theta}
	{\sqrt{1-k^2\sin^2\theta}},
	\qquad
	E(k):=\int_0^{\pi/2}\sqrt{1-k^2\sin^2\theta}\,d\theta,
\end{equation}
and their complementary-modulus counterparts
\begin{equation}
	K'(x):=K\!\left(\sqrt{1-x^2}\right),
	\qquad
	E'(x):=E\!\left(\sqrt{1-x^2}\right).
\end{equation}
For $\Re s>0$,
\begin{equation}
\begin{aligned}
M_{00}(s)&:=\int_0^1x^{2s-1}K'(x)^2\,\dd x,\\
M_{10}(s)=M_{01}(s)&:=\int_0^1x^{2s-1}E'(x)K'(x)\,\dd x,\\
M_{11}(s)&:=\int_0^1x^{2s-1}E'(x)^2\,\dd x.
\end{aligned}
\label{eq:moment-family}
\end{equation}
For $\varepsilon,\delta\in\{0,1\}$ define
\begin{equation}
Q_{\varepsilon,\delta}(s)
:=\left(s+\varepsilon+\delta,s,\frac12,\frac12;1;
 s+\frac12+\varepsilon,s+\frac12+\delta\right).
\label{eq:Q-family}
\end{equation}
The four points $Q_{00},Q_{10},Q_{01},Q_{11}$ will organize the three
moment families below.

\subsection{The $K'^2$ family}
\label{subsec:wan-k2}

We begin with the complementary quadratic moment $M_{00}(s)$ defined in
\eqref{eq:moment-family}.  Wan \cite{MR2845511} obtained a hypergeometric
evaluation of this family.  Written in our Mellin parameter $s$, it takes
the form stated below. Our aim is to recover the same formula from Mishev's
completed $L$-function. The main step is to identify $M_{00}(s)$ with
the specialization $L(Q_{00}(s))$; a $W(D_5)$ reflection then transforms
this specialization into the ${}_7F_6(1)$ representation appearing in
Wan's evaluation.

\begin{theorem}[Wan's $K'^2$ evaluation]
	\label{thm:wan-k2-mishev}
	For $\Re s>0$,
	\begin{equation}
		M_{00}(s)
		=
		\frac{\pi^2s\Gamma(s)^4}
		{8\Gamma(s+\frac12)^4}
		\F76\!\left(
		\begin{matrix}
			s,\ 1+\frac{s}{2},\ s,\
			\frac12,\frac12,\frac12,\frac12\\
			\frac{s}{2},\
			s+\frac12,s+\frac12,s+\frac12,s+\frac12,\ 1
		\end{matrix};1
		\right).
	\end{equation}
\end{theorem}

\begin{proof}
Set
\begin{equation}
	\gamma_n:=\left(\frac{(1/2)_n}{n!}\right)^2,
	\qquad
	H_n:=\sum_{k=1}^n\frac1k,
	\qquad
	d_n:=2(H_{2n}-H_n),
\end{equation}
with $H_0:=0$.
Cayley's complementary-modulus expansion,  written on p.~54 in his
notation for $F_1k$, becomes in the present
notation \cite[p.~54]{MR124532}
\begin{equation}
	K'(x)
	=
	\sum_{n=0}^{\infty}
	\gamma_n x^{2n}
	\left(
	\log\frac4x-d_n
	\right),
	\qquad 0<x<1.
	\label{eq:cayley-Kprime}
\end{equation}
Introduce the one-fold Mellin transform
	\begin{equation}
		J_K(u):=\int_0^1x^{2u-1}K'(x)\,\dd x.
	\end{equation}
	For $\Re u>0$, the beta integral followed by Gauss summation gives
\begin{equation}
		J_K(u)=\frac{\pi}{4}\frac{\Gamma(u)^2}{\Gamma(u+\frac12)^2},
		\label{eq:JK}
\end{equation}
a formula also recorded by Wan \cite[p.~126]{MR2845511}.
Consequently,
	\begin{equation}
		\frac{J_K'(u)}{J_K(u)}
		=
		2\psi(u)-2\psi\!\left(u+\frac12\right),
	\end{equation}
	where $\psi(u)=\frac{\Gamma'(u)}{\Gamma(u)}$ is the digamma function. Differentiating the defining integral for $J_K(u)$ with respect to $u$ gives
	\begin{equation}
		\frac12 J_K'(u)
		=
		\int_0^1 x^{2u-1}\log x\,K'(x)\,\dd x.
	\end{equation}
	Hence
	\begin{equation}
		\int_0^1 x^{2u-1}\log\frac4x\,K'(x)\,\dd x
		=
		\log4\,J_K(u)-\frac12J_K'(u).
		\label{eq:JK-log-moment}
	\end{equation}
	Substituting Cayley's expansion \eqref{eq:cayley-Kprime} for one copy
	of $K'(x)$ in $M_{00}(s)$ and using \eqref{eq:JK-log-moment} therefore gives
	\begin{equation}
	M_{00}(s)
	=
	J_K(s)\sum_{n=0}^{\infty}\phi_n(s)B_n(s),
	\end{equation}
	where
	\begin{equation}
	\phi_n(s)
	=
	\frac{(s)_n^2(1/2)_n^2}
	{(s+\frac12)_n^2(n!)^2},
	\end{equation}
	and
	\begin{equation}
	B_n(s)
	=
	\log4-\psi(s+n)
	+\psi\!\left(s+n+\frac12\right)
	-2(H_{2n}-H_n).
	\end{equation}
	Write
	\begin{equation}
	D_a(n):=\psi(a+n)-\psi(a).\label{eq:digamma-increment}	
	\end{equation}
	Since
	\begin{equation}
	2(H_{2n}-H_n)=D_{1/2}(n)-D_1(n),
	\end{equation}
	we obtain
	\begin{equation}
	B_n(s)
	=
	\beta_K(s)
	+D_{s+1/2}(n)-D_s(n)
	+D_1(n)-D_{1/2}(n),
	\end{equation}
	where
	\begin{equation}
		\beta_K(s)
		:=
		\log4-\frac12\frac{J_K'(s)}{J_K(s)}
		=
		\log4-\psi(s)+\psi\!\left(s+\frac12\right).
	\end{equation}
	The four digamma differences are naturally realized by a tangent
	direction on the Saalsch\"utzian hyperplane.  Define the balanced path
	\begin{equation}
	\mathfrak q_s(t)
	=
	\left(
	s-t,s,\frac12-t,\frac12;
	1-t;
	s+\frac12-t,s+\frac12
	\right)\label{eq:wan-q-path}	
	\end{equation}
	and
	\begin{equation}
	\Phi_s(t)
	:=
	\F43\!\left(
	\begin{matrix}
		s-t,\ s,\ \frac12-t,\ \frac12\\
		1-t,\ s+\frac12-t,\ s+\frac12
	\end{matrix};1
	\right).
	\end{equation}
If $\phi_n(s;t)$ denotes the $n$-th summand of $\Phi_s(t)$,
then $\phi_n(s;0)=\phi_n(s)$, and logarithmic differentiation gives
\begin{equation}
	\left.
	\frac{\dd}{\dd t}\log \phi_n(s;t)
	\right|_{t=0}
	=
	D_{s+1/2}(n)-D_s(n)
	+D_1(n)-D_{1/2}(n).
\end{equation}
Hence
\begin{equation}
	B_n(s)
	=
	\beta_K(s)
	+
	\left.
	\frac{\dd}{\dd t}\log \phi_n(s;t)
	\right|_{t=0}.
\end{equation}
Multiplying by $\phi_n(s)=\phi_n(s;0)$ gives
\begin{equation}
	\phi_n(s)B_n(s)
	=
	\beta_K(s)\phi_n(s)
	+
	\left.
	\frac{\dd}{\dd t}\phi_n(s;t)
	\right|_{t=0}.
\end{equation}
After summing over $n$, we arrive at
\begin{equation}
	M_{00}(s)
	=
	J_K(s)
	\left[
	\Phi_s'(0)+\beta_K(s)\Phi_s(0)
	\right].
	\label{eq:wan-k2-parameter-derivative}
\end{equation}

We now pass from the bare hypergeometric series to Mishev's completed
$L$-function.  Along the path $\mathfrak q_s(t)$, set
\begin{equation}
	\mathcal D_s(t)
	=
	\Gamma(1-t)
	\Gamma\!\left(s+\frac12-t\right)
	\Gamma\!\left(s+\frac12\right)
	\Gamma(s)
	\Gamma(s+t)
	\Gamma\!\left(\frac12\right)
	\Gamma\!\left(\frac12+t\right).
\end{equation}
A direct substitution into the two branches of Mishev's completed
$L$-function gives
\begin{equation}
	L(\mathfrak q_s(t))
	=
	\frac{1}{\sin(\pi t)}
	\left[
	\frac{\Phi_s(t)}{\mathcal D_s(t)}
	-
	\frac{\Phi_s(-t)}{\mathcal D_s(-t)}
	\right].
\end{equation}
This suggests introducing the completed hypergeometric branch
\begin{equation}
	\widehat{\Phi}_s(t)
	:=
	\frac{\Phi_s(t)}{\mathcal D_s(t)}.
\end{equation}
Thus
\begin{equation}
	L(\mathfrak q_s(t))
	=
	\frac{
		\widehat{\Phi}_s(t)-\widehat{\Phi}_s(-t)
	}{
		\sin(\pi t)
	}.
\end{equation}
Since $\mathfrak q_s(0)=Q_{00}(s)$, letting $t\to0$ yields
\begin{equation}
	L(Q_{00}(s))
	=
	\lim_{t\to0}
	\frac{
		\widehat{\Phi}_s(t)-\widehat{\Phi}_s(-t)
	}{
		\sin(\pi t)
	}
	=
	\frac{2}{\pi}\widehat{\Phi}_s'(0).
\end{equation}

At $t=0$,
\begin{equation}
	\mathcal D_s(0)
	=
	\pi\Gamma(s)^2
	\Gamma\!\left(s+\frac12\right)^2,
\end{equation}
while logarithmic differentiation gives
\begin{equation}
	-\frac{\mathcal D_s'(0)}{\mathcal D_s(0)}
	=
	\psi(1)
	+\psi\!\left(s+\frac12\right)
	-\psi(s)
	-\psi\!\left(\frac12\right).
\end{equation}
Using
\begin{equation}
	\psi(1)-\psi\!\left(\frac12\right)
	=
	2\log2=\log4,
\end{equation}
we obtain
\begin{equation}
	-\frac{\mathcal D_s'(0)}{\mathcal D_s(0)}
	=
	\beta_K(s).
\end{equation}
Consequently,
\begin{equation}
	\widehat{\Phi}_s'(0)
	=
	\frac{1}{\mathcal D_s(0)}
	\left[
	\Phi_s'(0)+\beta_K(s)\Phi_s(0)
	\right],
\end{equation}
and hence
\begin{equation}
	L(Q_{00}(s))
	=
	\frac{2}{\pi\mathcal D_s(0)}
	\left[
	\Phi_s'(0)+\beta_K(s)\Phi_s(0)
	\right].
\end{equation}	

Comparison with \eqref{eq:wan-k2-parameter-derivative}, together with
the evaluation of $J_K(s)$, yields
	\begin{equation}
	M_{00}(s)
	=
	\frac{\pi^3}{8}\Gamma(s)^4
	L(Q_{00}(s)).
	\label{eq:wan-k2-completed-L}
	\end{equation}
	This identity is the crucial completion step.  The term
	$\beta_K(s)$ arising from the complementary-modulus expansion is
	matched exactly by the logarithmic derivative of the gamma
	completion in Mishev's $L$-function.
	
	It remains to use the \(W(D_5)\) symmetry.  Mishev's fundamental
	reflection
	\begin{equation}
	L(a,b,c,d;e;f,g)
	=
	L(a,b,g-c,g-d;
	1+a+b-f;
	1+a+b-e,g)
	\end{equation}
	sends
	\begin{equation}
	Q_{00}(s)
	=
	\left(
	s,s,\frac12,\frac12;
	1;
	s+\frac12,s+\frac12
	\right)
	\end{equation}
	to
	\begin{equation}
	\mathfrak{p}_s
	=
	\left(
	s,s,s,s;
	s+\frac12;
	2s,s+\frac12
	\right).\label{eq:wan-p-point}
	\end{equation}
	Hence
	\begin{equation}
	L(Q_{00}(s))=L(\mathfrak{p}_s).
	\end{equation}
	At the reflected point $\mathfrak{p}_s$, the Bailey--Mishev
	very-well-poised representation has parameter
	\begin{equation}
	\alpha=d+g-e=s
	\end{equation}
	and gives
	\begin{equation}
	L(\mathfrak{p}_s)
	=
	\frac{s}{\pi\Gamma(s+\frac12)^4}
	\F76\!\left(
	\begin{matrix}
		s,\ 1+\frac{s}{2},\ s,\
		\frac12,\frac12,\frac12,\frac12\\[1mm]
		\frac{s}{2},\
		s+\frac12,s+\frac12,s+\frac12,s+\frac12,\ 1
	\end{matrix};1
	\right).
	\end{equation}
Substituting this expression into
\eqref{eq:wan-k2-completed-L} proves the theorem.
\end{proof}

\begin{remark}[Why the completion is essential]
	The preceding derivation explains a feature that is invisible if Wan's
	very-well-poised series is viewed merely as a closed formula.  The
	complementary-modulus expansion separates naturally into a
	parameter-dependent harmonic part and the $n$-independent contribution
	\begin{equation}
		\beta_K(s)
		=
		\log4-\psi(s)+\psi\!\left(s+\frac12\right).
	\end{equation}
	The harmonic part is realized by the tangent derivative of the balanced
	${}_4F_3(1)$ family $\Phi_s(t)$, while the $n$-independent contribution
	satisfies
	\begin{equation}
		\beta_K(s)
		=
		-\frac{\mathcal D_s'(0)}{\mathcal D_s(0)}.
	\end{equation}
	Thus the two pieces combine precisely under the gamma completion in
	Mishev's $L$-function.
\end{remark}

\begin{remark}[Wan's original parametrization]
	To compare with Wan's original form, put $m=2s-1$.  Applying the
	duplication formula for the gamma function to
	Theorem~\ref{thm:wan-k2-mishev} gives
	\begin{equation}
		\int_0^1 x^m K'(x)^2\,\dd x
		=
		\frac{2^{4m}(m+1)}{16}
		\frac{\Gamma\!\left(\frac{m+1}{2}\right)^8}
		{\Gamma(m+1)^4}
		{}_7F_6\!\left(
		\begin{matrix}
			\frac{m+1}{2},\ \frac{m+5}{4},\ \frac{m+1}{2},
			\ \frac12,\ \frac12,\ \frac12,\ \frac12\\
			\frac{m+1}{4},\ \frac{m+2}{2},\ \frac{m+2}{2},
			\ \frac{m+2}{2},\ \frac{m+2}{2},\ 1
		\end{matrix}
		;1
		\right),
	\end{equation}
	which is Wan's original parametrization \cite[Proposition~1(3), equation~(18), p.~125]{MR2845511}.
\end{remark}

\subsection{The $E'K'$ family}
\label{subsec:wan-ek}

We next consider the mixed complementary moment $M_{10}(s)$ defined in
\eqref{eq:moment-family}.  We first derive its representation in terms of
Mishev's completed $L$-function by the same complementary-modulus
collision mechanism used for the $K'^2$ family.

\subsubsection{The one-fold Mellin transform of $E'$}

Define
\begin{equation}
	J_E(u)
	:=
	\int_0^1 x^{2u-1}E'(x)\,\dd x.
\end{equation}
For $\Re u>0$, the beta integral followed by Gauss summation gives
\begin{equation}
	J_E(u)
	=
	\frac{\pi}{4}
	\frac{\Gamma(u)\Gamma(u+1)}
	{\Gamma(u+\frac12)\Gamma(u+\frac32)},
	\label{eq:JE}
\end{equation}
a formula also recorded by Wan \cite[p.~126]{MR2845511}.

Comparing \eqref{eq:JE} with \eqref{eq:JK}, we obtain
\begin{equation}
	J_E(u)
	=
	\frac{u}{u+\frac12}J_K(u).
	\label{eq:JE-JK}
\end{equation}
Thus the passage from $K'$ to $E'$ introduces a single contiguous
factor $u/(u+\frac12)$.

Differentiating the defining integral for $J_E(u)$ gives
\begin{equation}
	\frac12 J_E'(u)
	=
	\int_0^1 x^{2u-1}\log x\,E'(x)\,\dd x.
\end{equation}
Consequently,
\begin{equation}
	\int_0^1 x^{2u-1}\log\frac4x\,E'(x)\,\dd x
	=
	\log4\,J_E(u)-\frac12J_E'(u).
	\label{eq:JE-log-moment}
\end{equation}
\subsubsection{Cayley's expansion and the mixed moment}

Using Cayley's expansion \eqref{eq:cayley-Kprime} for $K'(x)$ together
with the logarithmic Mellin identity \eqref{eq:JE-log-moment}, we obtain
\begin{equation}
	M_{10}(s)
	=
	\sum_{n=0}^{\infty}
	\gamma_n J_E(s+n)\,C_n(s),
\end{equation}
where
\begin{equation}
	C_n(s)
	:=
	\log4
	-\frac12\frac{J_E'(s+n)}{J_E(s+n)}
	-2(H_{2n}-H_n).
\end{equation}

The shifted Mellin transform satisfies
\begin{equation}
	\frac{J_E(s+n)}{J_E(s)}
	=
	\frac{(s)_n(s+1)_n}
	{(s+\frac12)_n(s+\frac32)_n}.
	\label{eq:JE-shift}
\end{equation}
Accordingly, set
\begin{equation}
	\theta_n(s)
	:=
	\frac{\gamma_nJ_E(s+n)}{J_E(s)}
	=
	\frac{
		(s+1)_n(s)_n(\frac12)_n^2
	}{
		(s+\frac32)_n(s+\frac12)_n(n!)^2
	}.
\end{equation}
Thus $\theta_n(s)$ is the $n$-th coefficient of
\begin{equation}
	\Theta_s(0)
	:=
	{}_4F_3\!\left(
	\begin{matrix}
		s+1,\ s,\ \frac12,\ \frac12\\
		1,\ s+\frac32,\ s+\frac12
	\end{matrix};1
	\right),
\end{equation}
and
\begin{equation}
	M_{10}(s)
	=
	J_E(s)
	\sum_{n=0}^{\infty}
	\theta_n(s)C_n(s).
	\label{eq:EK-moment-series}
\end{equation}

Using the digamma increment $D_a(n)$ introduced above, we have
\begin{equation}
	D_1(n)-D_{1/2}(n)
	=
	-2(H_{2n}-H_n),
\end{equation}
while logarithmic differentiation of \eqref{eq:JE} gives
\begin{equation}
	\frac{J_E'(u)}{J_E(u)}
	=
	\psi(u)+\psi(u+1)
	-\psi\!\left(u+\frac12\right)
	-\psi\!\left(u+\frac32\right).
\end{equation}
It follows that
\begin{equation}
	C_n(s)
	=
	\beta_E(s)
	+D_1(n)-D_{1/2}(n)
	+\frac12
	\left[
	D_{s+\frac12}(n)+D_{s+\frac32}(n)
	-D_s(n)-D_{s+1}(n)
	\right],
	\label{eq:EK-Cn-decomposition}
\end{equation}
where
\begin{equation}
	\beta_E(s)
	:=
	\log4
	-\frac12\frac{J_E'(s)}{J_E(s)}
	=
	\log4
	-\frac12
	\left[
	\psi(s)+\psi(s+1)
	-\psi\!\left(s+\frac12\right)
	-\psi\!\left(s+\frac32\right)
	\right].
	\label{eq:beta-E}
\end{equation}

\subsubsection{The balanced collision path}

Consider the balanced path
\begin{equation}
	\mathfrak r_s(t)
	:=
	\left(
	s+1-\frac t2,\,
	s-\frac t2,\,
	\frac12-\frac t2,\,
	\frac12-\frac t2;
	1-t;\,
	s+\frac32-\frac t2,\,
	s+\frac12-\frac t2
	\right), \label{eq:EK-balanced-path}
	\end{equation}
so that
\[
\mathfrak r_s(0)=Q_{10}(s).
\]
The Saalsch\"utz condition is preserved identically in $t$.

Define
\begin{equation}
	\Theta_s(t)
	:=
	{}_4F_3\!\left(
	\begin{matrix}
		s+1-\frac t2,\,
		s-\frac t2,\,
		\frac12-\frac t2,\,
		\frac12-\frac t2\\
		1-t,\,
		s+\frac32-\frac t2,\,
		s+\frac12-\frac t2
	\end{matrix};1
	\right).
\end{equation}

If $\theta_n(s;t)$ denotes the $n$-th summand of $\Theta_s(t)$,
then $\theta_n(s;0)=\theta_n(s)$, and logarithmic differentiation gives
\begin{equation}
	\left.
	\frac{\dd}{\dd t}\log\theta_n(s;t)
	\right|_{t=0}
	=
	D_1(n)-D_{1/2}(n)
	+
	\frac12
	\left[
	D_{s+\frac12}(n)+D_{s+\frac32}(n)
	-D_s(n)-D_{s+1}(n)
	\right].
	\label{eq:EK-theta-logder}
\end{equation}

Comparing \eqref{eq:EK-theta-logder} with
\eqref{eq:EK-Cn-decomposition}, we obtain
\begin{equation}
	C_n(s)
	=
	\beta_E(s)
	+
	\left.
	\frac{\dd}{\dd t}\log\theta_n(s;t)
	\right|_{t=0}.
\end{equation}
Multiplying by $\theta_n(s)=\theta_n(s;0)$ gives
\begin{equation}
	\theta_n(s)C_n(s)
	=
	\beta_E(s)\theta_n(s)
	+
	\left.
	\frac{\dd}{\dd t}\theta_n(s;t)
	\right|_{t=0}.
\end{equation}

Using \eqref{eq:EK-moment-series} and summing over $n$, we obtain
\begin{equation}
	M_{10}(s)
	=
	J_E(s)
	\left[
	\Theta_s'(0)+\beta_E(s)\Theta_s(0)
	\right].
	\label{eq:EK-parameter-derivative}
\end{equation}

Thus the mixed elliptic moment is again expressed as a completed
parameter derivative of a balanced ${}_4F_3(1)$ family.

\subsubsection{The \(e=1\) collision in Mishev's \(L\)-function}

Along the balanced path \eqref{eq:EK-balanced-path}, the distinguished
lower parameter is \(e=1-t\), and hence
\[
\sin(\pi e)
=
\sin\bigl(\pi(1-t)\bigr)
=
\sin(\pi t).
\]

Substituting the parameters of \(\mathfrak r_s(t)\) into the gamma
denominator in Mishev's definition \eqref{eq:mishev-L}, we obtain the
normalization of the first branch
\begin{equation}
	\begin{aligned}
		\mathcal G_s(t):={}&
		\Gamma(1-t)
		\Gamma\!\left(s+\frac32-\frac t2\right)
		\Gamma\!\left(s+\frac12-\frac t2\right)
		\\
		&\times
		\Gamma\!\left(s+1+\frac t2\right)
		\Gamma\!\left(s+\frac t2\right)
		\Gamma\!\left(\frac12+\frac t2\right)^2.
	\end{aligned}
\end{equation}
Thus the first normalized \(\F43\) branch is
\[
\frac{\Theta_s(t)}{\mathcal G_s(t)}.
\]

For the supplementary branch, the four upper parameters are
\[
s+1+\frac t2,\qquad
s+\frac t2,\qquad
\frac12+\frac t2,\qquad
\frac12+\frac t2,
\]
while the three lower parameters are
\[
1+t,\qquad
s+\frac32+\frac t2,\qquad
s+\frac12+\frac t2.
\]
After permutation, the corresponding hypergeometric series is
\(\Theta_s(-t)\), and the accompanying gamma normalization is
\(\mathcal G_s(-t)\).  Therefore
\begin{equation}
	L(\mathfrak r_s(t))
	=
	\frac{1}{\sin(\pi t)}
	\left[
	\frac{\Theta_s(t)}{\mathcal G_s(t)}
	-
	\frac{\Theta_s(-t)}{\mathcal G_s(-t)}
	\right].
\end{equation}

It is convenient to introduce the completed hypergeometric branch
\begin{equation}
	\widehat{\Theta}_s(t)
	:=
	\frac{\Theta_s(t)}{\mathcal G_s(t)}.
\end{equation}
Then
\begin{equation}
	L(\mathfrak r_s(t))
	=
	\frac{
		\widehat{\Theta}_s(t)-\widehat{\Theta}_s(-t)
	}{
		\sin(\pi t)
	}.
\end{equation}
Since \(\mathfrak r_s(0)=Q_{10}(s)\), letting \(t\to0\) gives
\begin{equation}
	L(Q_{10}(s))
	=
	\frac{2}{\pi}\widehat{\Theta}_s'(0).
\end{equation}

At \(t=0\),
\begin{equation}
	\mathcal G_s(0)
	=
	\pi\,
	\Gamma(s)\Gamma(s+1)
	\Gamma\!\left(s+\frac12\right)
	\Gamma\!\left(s+\frac32\right).
\end{equation}
Moreover, logarithmic differentiation gives
\begin{equation}
	-\frac{\mathcal G_s'(0)}{\mathcal G_s(0)}
	=
	\beta_E(s).
\end{equation}
Consequently,
\begin{equation}
	\widehat{\Theta}_s'(0)
	=
	\frac{1}{\mathcal G_s(0)}
	\left[
	\Theta_s'(0)+\beta_E(s)\Theta_s(0)
	\right],
\end{equation}
and hence
\begin{equation}
	L(Q_{10}(s))
	=
	\frac{2}{\pi\mathcal G_s(0)}
	\left[
	\Theta_s'(0)+\beta_E(s)\Theta_s(0)
	\right].
	\label{eq:EK-completed-L}
\end{equation}

Comparing \eqref{eq:EK-completed-L} with
\eqref{eq:EK-parameter-derivative}, and using \eqref{eq:JE}, we obtain
\begin{equation}
	M_{10}(s)
	=
	\frac{\pi^3}{8}
	s^2\Gamma(s)^4
	L(Q_{10}(s)).
	\label{eq:EK-Mishev}
\end{equation}
Thus the mixed moment admits the same completed-\(L\) interpretation as
the \(K'^2\) family.

\subsubsection{The two very-well-poised representations}

Mishev's $L$-function is symmetric in the lower parameters $f$ and $g$:
\[
L(a,b,c,d;e;f,g)
=
L(a,b,c,d;e;g,f).
\]
By the definition of the points $Q_{\varepsilon,\delta}(s)$, interchanging
$f$ and $g$ sends $Q_{10}(s)$ to $Q_{01}(s)$.  Hence
\begin{equation}
	L(Q_{10}(s))
	=
	L(Q_{01}(s)).
	\label{eq:EK-fg-symmetry}
\end{equation}
This symmetry explains the two very-well-poised ${}_7F_6(1)$
representations of the same mixed moment obtained by Wan.

Applying the Bailey--Mishev representation
\eqref{eq:bailey-mishev} to $Q_{10}(s)$ gives
\begin{equation}
	L(Q_{10}(s))
	=
	\frac{1}
	{\pi\Gamma(s+\frac12)^3\Gamma(s+\frac32)}
	\,\mathcal W_{10}(s),
	\label{eq:EK-L-Q10}
\end{equation}
where
\begin{equation}
	\mathcal W_{10}(s)
	:=
	\F76\!\left(
	\begin{matrix}
		s,\ 1+\frac s2,\ s,\ -\frac12,\ \frac12,\ \frac12,\ \frac12\\
		\frac s2,\ 1,\ s+\frac12,\ s+\frac12,\ s+\frac12,\ s+\frac32
	\end{matrix};1
	\right).
	\label{eq:EK-W10}
\end{equation}
Combining \eqref{eq:EK-L-Q10} with \eqref{eq:EK-Mishev} yields
\begin{equation}
	M_{10}(s)
	=
	\frac{\pi^2s^2}{4(2s+1)}
	\frac{\Gamma(s)^4}{\Gamma(s+\frac12)^4}
	\,\mathcal W_{10}(s).
	\label{eq:EK-Wan-first}
\end{equation}
With Wan's parameter $n=2s-1$, this is precisely the first
hypergeometric representation of the mixed moment in
Wan's Proposition~1(2), equation~(17)
\cite[p.~125]{MR2845511}.
The prefactors agree by Legendre's duplication formula.

On the other hand, applying \eqref{eq:bailey-mishev} to
$Q_{01}(s)$ gives
\begin{equation}
	L(Q_{01}(s))
	=
	\frac{s(s+1)}
	{\pi\Gamma(s+\frac12)\Gamma(s+\frac32)^3}
	\,\mathcal W_{01}(s),
	\label{eq:EK-L-Q01}
\end{equation}
where
\begin{equation}
	\mathcal W_{01}(s)
	:=
	\F76\!\left(
	\begin{matrix}
		s+1,\ \frac{s+3}{2},\ s+1,\ \frac12,\ \frac12,\ \frac12,\ \frac32\\
		\frac{s+1}{2},\ 1,\ s+\frac12,\ s+\frac32,\ s+\frac32,\ s+\frac32
	\end{matrix};1
	\right).
	\label{eq:EK-W01}
\end{equation}
Using \eqref{eq:EK-Mishev} again, we obtain
\begin{equation}
	M_{10}(s)
	=
	\frac{\pi^2s^3(s+1)}
	{(2s+1)^3}
	\frac{\Gamma(s)^4}{\Gamma(s+\frac12)^4}
	\,\mathcal W_{01}(s).
	\label{eq:EK-Wan-second}
\end{equation}
Under the same substitution $n=2s-1$, this is the second
hypergeometric representation in Wan's Proposition~1(2), equation~(17)
\cite[p.~125]{MR2845511}.

Thus Wan's two apparently different very-well-poised evaluations arise
from the single symmetry \eqref{eq:EK-fg-symmetry}.  Comparing
\eqref{eq:EK-Wan-first} and \eqref{eq:EK-Wan-second} gives
\begin{equation}
	\mathcal W_{10}(s)
	=
	\frac{4s(s+1)}{(2s+1)^2}\,
	\mathcal W_{01}(s).
	\label{eq:EK-W-relation}
\end{equation}
Hence \eqref{eq:EK-W-relation} is the $f\leftrightarrow g$ symmetry of
Mishev's completed $L$-function transported to its two very-well-poised
${}_7F_6(1)$ realizations.

\subsection{The $E'^2$ family: direct collision--Cayley proof}
\label{subsec:wan-e2-collision}

The remaining vertex of the complementary-moment family is
\begin{equation}
Q_{11}(s)=\left(s+2,s,\frac12,\frac12;1;s+\frac32,s+\frac32\right).
\label{eq:Q11}
\end{equation}
We now derive the $E'^2$ moment directly from Cayley's
complementary-modulus expansion and the $W(D_5)$ collision mechanism.
We continue to use the coefficients $\gamma_n$ and $d_n$ introduced in
\eqref{eq:cayley-Kprime}.  The complementary derivative identity
\begin{equation}
	E'(x)
	=
	x^2K'(x)-x(1-x^2)\frac{\dd}{\dd x}K'(x)
	\label{eq:Eprime-from-Kprime}
\end{equation}
together with Cayley's expansion \eqref{eq:cayley-Kprime} for $K'$
gives \cite[p.~54]{MR124532}
\begin{equation}
	E'(x)
	=
	1+
	\sum_{n=1}^{\infty}
	\frac{2n}{2n-1}\,
	\gamma_n x^{2n}
	\left(
	\log\frac4x-d_n+\frac{1}{2n(2n-1)}
	\right),
	\qquad 0<x<1.
	\label{eq:cayley-Eprime}
\end{equation}

\subsubsection{Reduction to one-fold Mellin transforms}
\label{subsubsec:e2dc-mellin}

Substituting Cayley's expansion \eqref{eq:cayley-Eprime} into one
factor of the integrand defining $M_{11}(s)$ in
\eqref{eq:moment-family}, and using \eqref{eq:JE-log-moment}, gives
\begin{equation}
	M_{11}(s)
	=
	J_E(s)
	+
	\sum_{n=1}^{\infty}
	\alpha_n J_E(s+n)\,\mathcal B_n(s),
	\label{eq:moment-family-one-sum}
\end{equation}
where
\begin{equation}
	\alpha_n
	:=
	\frac{2n}{2n-1}\gamma_n
	\label{eq:e2dc-alpha}
\end{equation}
and
\begin{equation}
	\mathcal B_n(s)
	:=
	\log4-d_n+\frac{1}{2n(2n-1)}
	-\frac12\frac{J_E'(s+n)}{J_E(s+n)}.
	\label{eq:e2dc-Bn}
\end{equation}

The shift $n=m+1$ reveals the balanced hypergeometric coefficient.
Define
\begin{equation}
	p_m(s)
	:=
	\frac{
		(s+1)_m(s+2)_m(\frac12)_m(\frac32)_m
	}{
		(2)_m(s+\frac32)_m(s+\frac52)_m\,m!
	}.
	\label{eq:e2dc-pm}
\end{equation}
A direct Pochhammer calculation gives
\begin{equation}
	\alpha_{m+1}J_E(s+m+1)
	=
	\frac12 J_E(s+1)p_m(s).
	\label{eq:e2dc-coeff-shift}
\end{equation}
Hence, after the index shift $n=m+1$, the coefficients of the sum in
\eqref{eq:moment-family-one-sum} are governed by the balanced series
\begin{equation}
	\Ffourthree\!\left(
	\begin{matrix}
		s+1,\ s+2,\ \frac12,\ \frac32\\
		2,\ s+\frac32,\ s+\frac52
	\end{matrix};1
	\right).
	\label{eq:e2dc-balanced-43}
\end{equation}
\subsubsection{The reflected collision path}
\label{subsubsec:e2dc-reflection}

Consider the balanced collision path
\begin{equation}
	Q_{11}(s;t)
	=
	\left(
	s+2-\frac t2,\,
	s-\frac t2,\,
	\frac12-\frac t2,\,
	\frac12-\frac t2;
	1-t;\,
	s+\frac32-\frac t2,\,
	s+\frac32-\frac t2
	\right).
	\label{eq:e2dc-Q11-path}
\end{equation}
At $t=0$ this is $Q_{11}(s)$.  Before applying Mishev's involution
\eqref{eq:mishev-involution}, reorder the upper parameters as
\[
\left(
\frac12-\frac t2,\,
s+2-\frac t2,\,
s-\frac t2,\,
\frac12-\frac t2
\right).
\]
After this permutation, the involution $\mathcal A$ sends the point to
\begin{equation}
	R_s(t)
	=
	\left(
	\frac12-\frac t2,\,
	s+2-\frac t2,\,
	\frac32,\,
	s+1;
	2-\frac t2;\,
	s+\frac52,\,
	s+\frac32-\frac t2
	\right).
	\label{eq:e2dc-R-path}
\end{equation}
By the upper-parameter symmetry and Mishev's $W(D_5)$ invariance,
\begin{equation}
	L(Q_{11}(s;t))
	=
	L(R_s(t)).
	\label{eq:e2dc-L-invariance}
\end{equation}

Write the first hypergeometric branch of \eqref{eq:mishev-L} at
$R_s(t)$ as
\begin{equation}
	F_s(t)
	:=
	\Ffourthree\!\left(
	\begin{matrix}
		\frac12-\frac t2,\,
		s+2-\frac t2,\,
		\frac32,\,
		s+1\\
		2-\frac t2,\,
		s+\frac52,\,
		s+\frac32-\frac t2
	\end{matrix};1
	\right).
	\label{eq:e2dc-F}
\end{equation}
At $t=0$, this is precisely the balanced ${}_4F_3(1)$ appearing in
\eqref{eq:e2dc-balanced-43}.  Its gamma denominator is
\begin{equation}
	\begin{aligned}
		\Delta_1(t):={}&
		\Gamma\!\left(2-\frac t2\right)
		\Gamma\!\left(s+\frac52\right)
		\Gamma\!\left(s+\frac32-\frac t2\right)
		\Gamma\!\left(-\frac12\right)
		\\
		&\times
		\Gamma(s+1)
		\Gamma\!\left(\frac12+\frac t2\right)
		\Gamma\!\left(s+\frac t2\right).
	\end{aligned}
	\label{eq:e2dc-D1}
\end{equation}

The supplementary hypergeometric branch is
\begin{equation}
	\widetilde F_s(t)
	:=
	\Ffourthree\!\left(
	\begin{matrix}
		-\frac12,\,
		s+1,\,
		\frac12+\frac t2,\,
		s+\frac t2\\
		s+\frac32+\frac t2,\,
		s+\frac12,\,
		\frac t2
	\end{matrix};1
	\right),
	\label{eq:e2dc-S}
\end{equation}
with gamma denominator
\begin{equation}
	\begin{aligned}
		\Delta_2(t):={}&
		\Gamma\!\left(\frac12-\frac t2\right)
		\Gamma\!\left(s+2-\frac t2\right)
		\Gamma\!\left(\frac32\right)
		\Gamma(s+1)
		\\
		&\times
		\Gamma\!\left(s+\frac32+\frac t2\right)
		\Gamma\!\left(s+\frac12\right)
		\Gamma\!\left(\frac t2\right).
	\end{aligned}
	\label{eq:e2dc-D2}
\end{equation}

Since
\[
\sin\!\left(\pi\left(2-\frac t2\right)\right)
=
-\sin\frac{\pi t}{2},
\]
Mishev's completed $L$-function takes the form
\begin{equation}
	L(R_s(t))
	=
	-\frac{1}{\sin(\pi t/2)}
	\left[
	\frac{F_s(t)}{\Delta_1(t)}
	-
	\frac{\widetilde F_s(t)}{\Delta_2(t)}
	\right].
	\label{eq:e2dc-LR}
\end{equation}
\subsubsection{The degenerate supplementary branch}
\label{subsubsec:e2dc-degeneration}

Set
\begin{equation}
	H_1(t):=\frac{F_s(t)}{\Delta_1(t)},
	\qquad
	H_2(t):=\frac{\widetilde F_s(t)}{\Delta_2(t)}.
	\label{eq:e2dc-H12}
\end{equation}
The lower parameter $t/2$ in $\widetilde F_s(t)$ is compensated by the
factor $1/\Gamma(t/2)$ in $1/\Delta_2(t)$.  Indeed, for $n\ge1$,
\begin{equation}
	\frac{1}{\Gamma(t/2)(t/2)_n}
	=
	\frac{1}{\Gamma(n+t/2)},
	\label{eq:e2dc-gamma-pairing}
\end{equation}
which is regular at $t=0$.  The normalized $n=0$ term behaves
differently, since $1/\Gamma(t/2)=O(t)$.

Write
\begin{equation}
	H_1(t)=\sum_{m=0}^{\infty}u_m(t),
	\qquad
	H_2(t)=v_0(t)+\sum_{m=0}^{\infty}v_{m+1}(t),
	\label{eq:e2dc-term-decomp}
\end{equation}
where $u_m(t)$ and $v_n(t)$ denote the normalized terms of the first
and supplementary branches, respectively.

\begin{lemma}[Matching of the shifted coefficients]
	\label{lem:e2dc-base-match}
	For every $m\ge0$,
	\begin{equation}
		v_{m+1}(0)=u_m(0).
		\label{eq:e2dc-base-match}
	\end{equation}
	Moreover, if
	\begin{equation}
		\mathcal N_s
		:=
		\frac{\pi^2}{4}
		\Gamma(s)\Gamma(s+1)^2\Gamma(s+2),
		\label{eq:e2dc-Ns}
	\end{equation}
	then
	\begin{equation}
		\mathcal N_su_m(0)
		=
		-\alpha_{m+1}J_E(s+m+1).
		\label{eq:e2dc-first-coeff-match}
	\end{equation}
\end{lemma}

\begin{proof}
	At $t=0$, the first branch \eqref{eq:e2dc-F} is precisely the balanced
	series \eqref{eq:e2dc-balanced-43}, so its $m$th series coefficient is
	$p_m(s)$.  Moreover,
	\[
	\Delta_1(0)
	=
	-2\pi\,
	\Gamma(s)\Gamma(s+1)
	\Gamma\!\left(s+\frac32\right)
	\Gamma\!\left(s+\frac52\right).
	\]
	Hence
	\begin{equation}
		\mathcal N_su_m(0)
		=
		-\frac{\pi}{8}
		\frac{\Gamma(s+1)\Gamma(s+2)}
		{\Gamma(s+\frac32)\Gamma(s+\frac52)}
		p_m(s)
		=
		-\frac12J_E(s+1)p_m(s).
	\end{equation}
	The second assertion therefore follows from
	\eqref{eq:e2dc-coeff-shift}.  On the other hand, using
	\eqref{eq:e2dc-gamma-pairing} together with
	$(a)_{m+1}=a(a+1)_m$ gives
	\[
	v_{m+1}(0)=u_m(0).
	\]
\end{proof}

\begin{lemma}[The isolated $n=0$ term]
	\label{lem:e2dc-n0}
	The supplementary $n=0$ term satisfies $v_0(0)=0$ and
	\begin{equation}
		\mathcal N_sv_0'(0)=J_E(s).
		\label{eq:e2dc-n0-match}
	\end{equation}
\end{lemma}

\begin{proof}
	Since
	\[
	\frac{1}{\Gamma(t/2)}
	=
	\frac t2+O(t^2),
	\]
	we have
	\begin{equation}
		v_0'(0)
		=
		\frac{1}
		{\pi\Gamma(s+1)\Gamma(s+2)
			\Gamma(s+\frac12)\Gamma(s+\frac32)}.
	\end{equation}
	Multiplying by \eqref{eq:e2dc-Ns} and using \eqref{eq:JE} at $u=s$
	gives \eqref{eq:e2dc-n0-match}.
\end{proof}

The preceding lemmas show that
\begin{equation}
	H_1(0)=H_2(0).
	\label{eq:e2dc-H-collision}
\end{equation}
Consequently, the apparent singularity in \eqref{eq:e2dc-LR} at
$t=0$ is removable.  Using \eqref{eq:e2dc-L-invariance}, we obtain
\begin{equation}
	L(Q_{11}(s))
	=
	L(R_s(0))
	=
	\frac{2}{\pi}
	\bigl(H_2'(0)-H_1'(0)\bigr).
	\label{eq:e2dc-L-derivative}
\end{equation}

\subsubsection{Matching the logarithmic derivatives}
\label{subsubsec:e2dc-logder}

For $m\ge0$, put $n=m+1$ and define
\begin{equation}
	\ell_m^{(1)}
	:=
	\frac{u_m'(0)}{u_m(0)},
	\qquad
	\ell_n^{(2)}
	:=
	\frac{v_n'(0)}{v_n(0)}.
\end{equation}
Logarithmic differentiation of the Pochhammer and gamma factors gives
\begin{equation}
	\begin{aligned}
		\ell_m^{(1)}
		=
		\frac12\bigg[
		&-\psi\!\left(n-\frac12\right)
		-\psi(s+n+1)
		+\psi(s+2)
		+\psi(n+1)
		\\
		&+\psi\!\left(s+n+\frac12\right)
		-\psi(s)
		\bigg],
	\end{aligned}
	\label{eq:e2dc-log-first}
\end{equation}
and
\begin{equation}
	\begin{aligned}
		\ell_n^{(2)}
		=
		\frac12\bigg[
		&\psi\!\left(n+\frac12\right)
		+\psi(s+n)
		-\psi(s)
		-\psi\!\left(s+n+\frac32\right)
		\\
		&+\psi(s+2)
		-\psi(n)
		\bigg].
	\end{aligned}
	\label{eq:e2dc-log-second}
\end{equation}
Subtracting gives
\begin{equation}
	\begin{aligned}
		\ell_n^{(2)}-\ell_m^{(1)}
		={}&
		\frac12\bigg[
		\psi\!\left(n-\frac12\right)
		+\psi\!\left(n+\frac12\right)
		-\psi(n)-\psi(n+1)
		\bigg]
		\\
		&+
		\frac12\bigg[
		\psi(s+n)+\psi(s+n+1)
		-\psi\!\left(s+n+\frac12\right)
		-\psi\!\left(s+n+\frac32\right)
		\bigg].
	\end{aligned}
	\label{eq:e2dc-log-diff-raw}
\end{equation}
By logarithmic differentiation of \eqref{eq:JE}, the second bracket is
\[
\frac{J_E'(s+n)}{J_E(s+n)}.
\]
For the first bracket, using
\[
\psi\!\left(n+\frac12\right)-\psi(n+1)
=
-\log4+d_n
\]
together with the recurrence
$\psi(z+1)=\psi(z)+1/z$, we obtain
\begin{equation}
	\frac12\bigg[
	\psi\!\left(n-\frac12\right)
	+\psi\!\left(n+\frac12\right)
	-\psi(n)-\psi(n+1)
	\bigg]
	=
	-\log4+d_n-\frac{1}{2n(2n-1)}.
	\label{eq:e2dc-half-digamma}
\end{equation}
Therefore
\begin{equation}
	\ell_n^{(2)}-\ell_{n-1}^{(1)}
	=
	-\mathcal B_n(s).
	\label{eq:e2dc-log-match}
\end{equation}
Thus the logarithmic derivative produced by the reflected collision
matches exactly the bracket $\mathcal B_n(s)$ arising from Cayley's
expansion of $E'$.

\subsubsection{Completion of the direct proof}

From \eqref{eq:e2dc-base-match} and \eqref{eq:e2dc-log-match},
\begin{equation}
	v_{m+1}'(0)-u_m'(0)
	=
	-u_m(0)\mathcal B_{m+1}(s).
\end{equation}
Multiplying by $\mathcal N_s$ and using
\eqref{eq:e2dc-first-coeff-match} gives
\begin{equation}
	\mathcal N_s
	\bigl(v_{m+1}'(0)-u_m'(0)\bigr)
	=
	\alpha_{m+1}J_E(s+m+1)\mathcal B_{m+1}(s).
\end{equation}
Together with \eqref{eq:e2dc-n0-match}, summation over $m\ge0$ yields
\begin{align}
	\mathcal N_s\bigl(H_2'(0)-H_1'(0)\bigr)
	&=
	J_E(s)
	+
	\sum_{n=1}^{\infty}
	\alpha_nJ_E(s+n)\mathcal B_n(s)
	\\
	&=
	M_{11}(s),
\end{align}
where the last equality is \eqref{eq:moment-family-one-sum}.
Using \eqref{eq:e2dc-L-derivative}, we therefore obtain
\begin{equation}
	M_{11}(s)
	=
	\frac{\pi}{2}\mathcal N_s\,L(Q_{11}(s)).
\end{equation}
Since
\[
\frac{\pi}{2}\mathcal N_s
=
\frac{\pi^3}{8}
\Gamma(s)\Gamma(s+1)^2\Gamma(s+2),
\]
we arrive at the following formula.

\begin{theorem}[$E'^2$ moment formula]
	\label{thm:e2-direct}
	For $\Re s>0$,
	\begin{equation}
		M_{11}(s)
		=
		\int_0^1x^{2s-1}E'(x)^2\,\dd x
		=
		\frac{\pi^3}{8}
		\Gamma(s)\Gamma(s+1)^2\Gamma(s+2)
		L(Q_{11}(s)).
	\end{equation}
\end{theorem}
The $E'^2$ member of the complementary quadratic moment family therefore
admits a direct Mishev collision proof parallel to the $K'^2$ and $E'K'$
cases.  The new feature is the index-shifting degeneration of the
supplementary branch along the reflected path $R_s(t)$.

\subsubsection{The very-well-poised form}
\label{subsubsec:e2dc-vwp}
Applying the Bailey--Mishev representation
\eqref{eq:bailey-mishev} at $Q_{11}(s)$ gives
\begin{equation}
	L(Q_{11}(s))
	=
	\frac{\Gamma(s+2)}
	{\pi\Gamma(s+\frac32)^2
		\Gamma(s+1)\Gamma(s+\frac52)\Gamma(s+\frac12)}
	\,\mathcal W_{11}(s),
\end{equation}
where
\begin{equation}
	\mathcal W_{11}(s)
	:=
	\F76\!\left(
	\begin{matrix}
		s+1,\ \frac{s+3}{2},\ s+1,\ -\frac12,\ \frac12,\ \frac12,\ \frac32\\
		\frac{s+1}{2},\ 1,\ s+\frac12,\ s+\frac32,\ s+\frac32,\ s+\frac52
	\end{matrix};1
	\right).
\end{equation}
Combining this with Theorem~\ref{thm:e2-direct} yields
\begin{equation}
	M_{11}(s)
	=
	\frac{2\pi^2s^3(s+1)^2}
	{(2s+1)^3(2s+3)}
	\frac{\Gamma(s)^4}{\Gamma(s+\frac12)^4}
	\,\mathcal W_{11}(s).
\end{equation}
With Wan's parameter $n=2s-1$, this is the very-well-poised
representation of the complementary $E'^2$ moment in
Wan's Proposition~1(1), equation~(16)
\cite[p.~125]{MR2845511}.
\subsection{The contiguous structure and a Barnes realization}
\label{subsec:wan-contiguous}

We now give a second route to the $E'^2$ moment formula.  The idea is
to construct two parallel contiguous systems.  On the elliptic-integral
side, the differential relations among $K'$ and $E'$ induce, after
taking Mellin moments, relations among $M_{00}$, $M_{10}$, and $M_{11}$.
On the hypergeometric side, we will derive the corresponding contiguous
relations among neighboring values of Mishev's completed $L$-function
directly from its Barnes representation.  Matching the two systems
will recover the $E'^2$ evaluation from the already established
$K'^2$ and $E'K'$ cases.

Recall the moment family \eqref{eq:moment-family}.  Using the
complementary derivative identity \eqref{eq:Eprime-from-Kprime} and
multiplying by $K'(x)$ gives
\begin{equation}
	E'(x)K'(x)
	=
	x^2K'(x)^2
	-
	\frac{x(1-x^2)}{2}
	\frac{\dd}{\dd x}K'(x)^2.
\end{equation}
Multiplying by $x^{2s-1}$, integrating over $(0,1)$, and integrating
by parts, we obtain
\begin{equation}
	M_{10}(s)
	=
	s\bigl(M_{00}(s)-M_{00}(s+1)\bigr).
	\label{eq:wan-M10-contiguous}
\end{equation}
Thus the mixed moment is the first contiguous descendant of the
$K'^2$ Mellin family.

Using the previously established $L$-function representations
\eqref{eq:wan-k2-completed-L} and \eqref{eq:EK-Mishev} for
$M_{00}$ and $M_{10}$, respectively,
\eqref{eq:wan-M10-contiguous} becomes
\begin{equation}
	L(Q_{00}(s))
	-
	s^4L(Q_{00}(s+1))
	=
	s\,L(Q_{10}(s)).
	\label{eq:wan-L10-contiguous}
\end{equation}
Hence the differential relation between $K'$ and $E'$ already induces
a nontrivial contiguous relation among neighboring values of Mishev's
completed $L$-function.

There is a second relation of the same type.  Set
\[
A(x):=K'(x)^2,\qquad
B(x):=E'(x)K'(x),\qquad
C(x):=E'(x)^2.
\]
The complementary derivative identities imply
\[
(1-x^2)B'(x)
=
xA(x)-\frac{1}{x}C(x).
\]
Multiplying by $x^{2s}$ and integrating over $(0,1)$ gives
\[
\int_0^1x^{2s}(1-x^2)B'(x)\,\dd x
=
M_{00}(s+1)-M_{11}(s).
\]
For $\Re s>0$, the boundary term in integration by parts vanishes, and
therefore
\[
\begin{aligned}
	\int_0^1x^{2s}(1-x^2)B'(x)\,\dd x
	&=
	-\int_0^1
	\left[
	2s x^{2s-1}(1-x^2)-2x^{2s+1}
	\right]B(x)\,\dd x
	\\
	&=
	-2sM_{10}(s)+2(s+1)M_{10}(s+1).
\end{aligned}
\]
Consequently,
\begin{equation}
	M_{11}(s)
	=
	M_{00}(s+1)
	+
	2sM_{10}(s)
	-
	2(s+1)M_{10}(s+1).
	\label{eq:wan-M11-contiguous}
\end{equation}

The two relations
\eqref{eq:wan-M10-contiguous} and
\eqref{eq:wan-M11-contiguous} organize the complementary moments into
the contiguous hierarchy
\[
M_{00}
\quad\longrightarrow\quad
M_{10}
\quad\longrightarrow\quad
M_{11}.
\]
The first arrow is already mirrored on the Mishev side by
\eqref{eq:wan-L10-contiguous}.  Our remaining task is to derive the
corresponding $L$-function relation for the second arrow.  Mishev's
Barnes representation provides the natural framework for doing so.

\subsubsection{The normalized \(2\times2\) Mishev family}
\label{subsubsec:wan-normalized-square}

Recall the four points $Q_{\varepsilon,\delta}(s)$ defined in
\eqref{eq:Q-family}.  To absorb the gamma factors appearing in the
moment formulas, set
\begin{equation}
	\Lambda_{\varepsilon,\delta}(s)
	:=
	\Gamma(s)\Gamma(s+\varepsilon)
	\Gamma(s+\delta)\Gamma(s+\varepsilon+\delta)
	L(Q_{\varepsilon,\delta}(s)).
	\label{eq:wan-Lambda-def}
\end{equation}
In particular,
\begin{equation}
	\Lambda_{00}(s)
	=
	\Gamma(s)^4L(Q_{00}(s)),
	\qquad
	\Lambda_{10}(s)
	=
	s^2\Gamma(s)^4L(Q_{10}(s)),
	\label{eq:wan-Lambda-0010}
\end{equation}
and, by the $f\leftrightarrow g$ symmetry,
\begin{equation}
	\Lambda_{01}(s)=\Lambda_{10}(s).
\end{equation}
Moreover,
\begin{equation}
	\Lambda_{11}(s)
	=
	s^3(s+1)\Gamma(s)^4L(Q_{11}(s)).
	\label{eq:wan-Lambda-11}
\end{equation}

The formulas already established for the first two moment families become
\begin{equation}
	M_{00}(s)
	=
	\frac{\pi^3}{8}\Lambda_{00}(s),
	\qquad
	M_{10}(s)
	=
	\frac{\pi^3}{8}\Lambda_{10}(s).
	\label{eq:wan-known-moments-Lambda}
\end{equation}
Accordingly, in view of \eqref{eq:wan-M11-contiguous}, it remains to prove
\begin{equation}
	\Lambda_{11}(s)
	=
	\Lambda_{00}(s+1)
	+
	2s\Lambda_{10}(s)
	-
	2(s+1)\Lambda_{10}(s+1).
	\label{eq:wan-Lambda-target}
\end{equation}
This is the normalized Mishev-side counterpart of the contiguous relation
for $M_{11}(s)$.  To prove \eqref{eq:wan-Lambda-target}, we express all
four terms by Mishev's Barnes representation over a common kernel.
Their prescribed linear combination then reduces to a single Barnes
integral whose integrand is an explicit unit-shift difference
$G_s(t+1)-G_s(t)$.  Integrating this identity and shifting the Barnes
contour by one unit, with no poles crossed and the boundary contributions
vanishing at infinity, shows that the integral is zero.  This proves
\eqref{eq:wan-Lambda-target}.

\subsubsection{Mishev's Barnes representation and a common kernel}
\label{subsubsec:wan-common-barnes-kernel}

We use Mishev's Barnes representation~\eqref{eq:mishev-barnes}, with Barnes contour $\mathfrak C$.

At the points \(Q_{\varepsilon,\delta}(s)\), one has
\[
e=1,\qquad c=d=\frac12,
\]
so \(\Gamma(1-e-t)\Gamma(-t)=\Gamma(-t)^2\).  Introduce
\begin{equation}
{\mathcal K_s(t):=\frac{\Gamma(s+t)^2\Gamma(\frac12+t)^2\Gamma(-t)^2}{\Gamma(s+\frac12+t)^2}.}
\label{eq:wan-barnes-kernel}
\end{equation}
For convenience, write
\[
u:=s+t.
\]
Substituting $Q_{\varepsilon,\delta}(s)$ into
\eqref{eq:mishev-barnes}, and using the normalization
\eqref{eq:wan-Lambda-def}, gives the common form
\begin{equation}
	\Lambda_{\varepsilon,\delta}(s)
	=
	\frac{1}{\pi^3}\frac{1}{2\pi i}
	\int_{\mathfrak C}
	\mathcal K_s(t)\,
	R_{\varepsilon,\delta}(s,u)\,\dd t,
	\label{eq:wan-barnes-master}
\end{equation}
where
\begin{equation}
	R_{\varepsilon,\delta}(s,u)
	=
	\frac{\Gamma(s+\varepsilon)\Gamma(s+\delta)}
	{\Gamma(s)\Gamma(s+\varepsilon+\delta)}
	\frac{\Gamma(u+\varepsilon+\delta)}{\Gamma(u)}
	\frac{\Gamma(u+\frac12)^2}
	{\Gamma(u+\frac12+\varepsilon)
		\Gamma(u+\frac12+\delta)}.
	\label{eq:wan-barnes-rational-factor}
\end{equation}
Since $\varepsilon,\delta\in\{0,1\}$, the gamma recurrence immediately
reduces these factors to
\[
R_{00}(s,u)=1,
\qquad
R_{10}(s,u)=R_{01}(s,u)=\frac{u}{u+\frac12},
\]
and
\[
R_{11}(s,u)
=
\frac{s}{s+1}
\frac{u(u+1)}{(u+\frac12)^2}.
\]
Consequently,
\begin{equation}
{\Lambda_{00}(s)=\frac{1}{\pi^3}\frac{1}{2\pi i}\int_{\mathfrak C}\mathcal K_s(t)\,\dd t,}
\label{eq:wan-barnes-Lambda00}
\end{equation}
\begin{equation}
{\Lambda_{10}(s)=\frac{1}{\pi^3}\frac{1}{2\pi i}\int_{\mathfrak C}\mathcal K_s(t)\frac{u}{u+\frac12}\,\dd t,}
\label{eq:wan-barnes-Lambda10}
\end{equation}
and
\begin{equation}
{\Lambda_{11}(s)=\frac{1}{\pi^3}\frac{1}{2\pi i}\int_{\mathfrak C}\mathcal K_s(t)\frac{s}{s+1}\frac{u(u+1)}{(u+\frac12)^2}\,\dd t.}
\label{eq:wan-barnes-Lambda11}
\end{equation}

\subsubsection{Shifted kernel relations}
\label{subsubsec:wan-shifted-kernel}

The gamma recurrence gives
\begin{equation}
\frac{\mathcal K_{s+1}(t)}{\mathcal K_s(t)}=\frac{u^2}{(u+\frac12)^2}.
\label{eq:wan-kernel-s-shift}
\end{equation}
Consequently,
\begin{equation}
{\Lambda_{00}(s+1)=\frac{1}{\pi^3}\frac{1}{2\pi i}\int_{\mathfrak C}\mathcal K_s(t)\frac{u^2}{(u+\frac12)^2}\,\dd t.}
\label{eq:wan-barnes-Lambda00-shift}
\end{equation}
Likewise,
\begin{equation}
{\Lambda_{10}(s+1)=\frac{1}{\pi^3}\frac{1}{2\pi i}\int_{\mathfrak C}\mathcal K_s(t)\frac{u^2(u+1)}{(u+\frac12)^2(u+\frac32)}\,\dd t.}
\label{eq:wan-barnes-Lambda10-shift}
\end{equation}
Thus all four quantities in \eqref{eq:wan-Lambda-target} are represented
over the same Barnes kernel $\mathcal K_s(t)$, with the differences
encoded entirely by elementary rational functions of $u=s+t$.
Accordingly, the target relation reduces to a single rational identity
inside the Barnes integral, up to a unit shift in the integration variable.

\subsubsection{The Barnes telescoper}
\label{subsubsec:wan-barnes-telescoper}

Subtract the right-hand side of \eqref{eq:wan-Lambda-target} from the left.
Using \eqref{eq:wan-barnes-Lambda10},
\eqref{eq:wan-barnes-Lambda11},
\eqref{eq:wan-barnes-Lambda00-shift}, and
\eqref{eq:wan-barnes-Lambda10-shift}, the difference is
\begin{equation}
	\frac{1}{\pi^3}\frac{1}{2\pi i}
	\int_{\mathfrak C}
	\mathcal K_s(t)\mathcal E_s(t)\,\dd t,
	\label{eq:wan-barnes-difference-integral}
\end{equation}
where
\begin{equation}
	\begin{aligned}
		\mathcal E_s(t)
		={}&
		\frac{s}{s+1}\frac{u(u+1)}{(u+\frac12)^2}
		-\frac{u^2}{(u+\frac12)^2}
		-2s\frac{u}{u+\frac12}
		\\
		&\quad
		+2(s+1)
		\frac{u^2(u+1)}
		{(u+\frac12)^2(u+\frac32)}.
	\end{aligned}
	\label{eq:wan-defect-factor}
\end{equation}

To construct a unit-shift telescoper, seek a rational function
$\rho_s(t)$ such that
\begin{equation}
	G_s(t):=\mathcal K_s(t)\rho_s(t)
\end{equation}
satisfies
\begin{equation}
	G_s(t+1)-G_s(t)
	=
	\mathcal K_s(t)\mathcal E_s(t).
\end{equation}
Dividing by $\mathcal K_s(t)$ reduces this to the rational difference
equation
\begin{equation}
	\frac{\mathcal K_s(t+1)}{\mathcal K_s(t)}
	\rho_s(t+1)-\rho_s(t)
	=
	\mathcal E_s(t).
\end{equation}

The gamma recurrence gives
\begin{equation}
	\frac{\mathcal K_s(t+1)}{\mathcal K_s(t)}
	=
	\frac{u^2(t+\frac12)^2}
	{(t+1)^2(u+\frac12)^2},
	\qquad u=s+t.
\end{equation}
The pole structure and the factor $(t+1)^{-2}$ suggest taking
\begin{equation}
	\rho_s(t)
	=
	C_s\frac{t^2u}{u+\frac12}.
\end{equation}
Substitution determines $C_s=2/(s+1)$.  Hence
\begin{equation}
	G_s(t)
	=
	\frac{4t^2(s+t)}
	{(s+1)(2s+2t+1)}
	\mathcal K_s(t).
	\label{eq:wan-G-def}
\end{equation}
A direct substitution gives
\begin{equation}
	G_s(t+1)-G_s(t)
	=
	\mathcal K_s(t)\mathcal E_s(t).
	\label{eq:wan-barnes-telescoper}
\end{equation}
Thus the defect does not vanish pointwise; rather, it is a discrete coboundary for the unit-shift operator at the Barnes-integrand level.
\subsubsection{Contour shift and completion of the contiguous relation}
\label{subsubsec:wan-barnes-contour-shift}

Choose a vertical Barnes contour
\begin{equation}
	\mathfrak C:\qquad
	\Re t=c,
	\qquad
	-\min(\Re s,\tfrac12)<c<0.
	\label{eq:wan-barnes-contour}
\end{equation}
The poles of $\Gamma(s+t)^2$ and $\Gamma(\frac12+t)^2$ lie to the left
of $\mathfrak C$, whereas the poles of $\Gamma(-t)^2$ at
$t=0,1,2,\ldots$ lie to the right.

Integrating \eqref{eq:wan-barnes-telescoper} over $\mathfrak C$ gives
\begin{equation}
	\int_{\mathfrak C}\mathcal K_s(t)\mathcal E_s(t)\,\dd t
	=
	\int_{\mathfrak C+1}G_s(t)\,\dd t
	-
	\int_{\mathfrak C}G_s(t)\,\dd t.
\end{equation}
No pole of $G_s$ lies in the strip between $\mathfrak C$ and
$\mathfrak C+1$.  The double pole of $\Gamma(-t)^2$ at $t=0$ is
cancelled by the factor $t^2$ in \eqref{eq:wan-G-def}; the apparent
pole at $s+t=-\frac12$ is cancelled by the double zero of
$1/\Gamma(s+\frac12+t)^2$; and the next pole of $\Gamma(-t)^2$, at
$t=1$, lies strictly to the right of $\mathfrak C+1$.  Stirling's formula gives exponential decay on vertical lines, so the integrals over the horizontal sides of the usual rectangular contour
tend to zero as the height tends to infinity.  Hence
\begin{equation}
	\int_{\mathfrak C}\mathcal K_s(t)\mathcal E_s(t)\,\dd t=0.
	\label{eq:wan-barnes-zero}
\end{equation}
By \eqref{eq:wan-barnes-difference-integral}, this proves
\begin{equation}
	\Lambda_{11}(s)
	=
	\Lambda_{00}(s+1)
	+
	2s\Lambda_{10}(s)
	-
	2(s+1)\Lambda_{10}(s+1).
	\label{eq:wan-Lambda-contiguous}
\end{equation}

\subsubsection{Completion of the Barnes proof}
\label{subsubsec:wan-e2-completion}

Combining \eqref{eq:wan-M11-contiguous},
\eqref{eq:wan-known-moments-Lambda}, and
\eqref{eq:wan-Lambda-contiguous}, we obtain
\begin{equation}
	M_{11}(s)
	=
	\frac{\pi^3}{8}\Lambda_{11}(s).
\end{equation}
Using \eqref{eq:wan-Lambda-11}, this becomes
\begin{equation}
	M_{11}(s)
	=
	\frac{\pi^3}{8}
	\Gamma(s)\Gamma(s+1)^2\Gamma(s+2)
	L(Q_{11}(s)),
	\label{eq:wan-e2-L-gamma}
\end{equation}
This provides an independent derivation of the $E'^2$ moment formula
in Theorem~\ref{thm:e2-direct}.
Equivalently, the normalized contiguous relation
\eqref{eq:wan-Lambda-contiguous} takes the unnormalized form
\begin{equation}
	(s+1)L(Q_{11}(s))
	=
	sL(Q_{00}(s+1))
	+
	2L(Q_{10}(s))
	-
	2s(s+1)^3L(Q_{10}(s+1)).
	\label{eq:wan-L-contiguous-unnormalized}
\end{equation}
Thus the classical elliptic differential system induces a genuine
contiguous relation among neighboring values of Mishev's completed
$L$-function.

\subsection{The $2\times2$ contiguous square}
\label{subsec:parameter-square}

The four points $Q_{\varepsilon,\delta}(s)$ defined in
\eqref{eq:Q-family} form the square
\begin{equation}
	\begin{matrix}
		Q_{00}(s)&\longrightarrow&Q_{10}(s)\\
		\downarrow&&\downarrow\\
		Q_{01}(s)&\longrightarrow&Q_{11}(s).
	\end{matrix}
	\label{eq:contiguous-square}
\end{equation}
The horizontal and vertical moves correspond to replacing one of the two
elliptic factors $K'$ by $E'$.  The exchange
\[
(\varepsilon,\delta)\longleftrightarrow(\delta,\varepsilon)
\]
corresponds on the Mishev side to the symmetry $f\leftrightarrow g$.
Thus $Q_{00}$ and $Q_{11}$ are fixed by this involution, whereas
$Q_{10}$ and $Q_{01}$ form a two-element orbit.  In particular, the
identity $E'K'=K'E'$ is reflected by the elementary Weyl-group symmetry
$f\leftrightarrow g$, which also accounts for the two very-well-poised
realizations of the mixed moment.

On the elliptic side, the three moment families form the quadratic system
\[
K'^2,\qquad E'K',\qquad E'^2,
\]
while on the hypergeometric side they correspond to the vertices of the
$2\times2$ family $Q_{\varepsilon,\delta}(s)$ inside the Saalsch\"utzian
parameter space, modulo the exchange symmetry.  Mishev's completed
$L$-function provides the common normalization, while its collision
structure and $W(D_5)$ symmetry organize the parameter geometry.
Accordingly, the three complementary moment families should be viewed as
different vertices of a single $W(D_5)$-compatible contiguous square.

\section{The BBBG four-Bessel moment}
\label{sec:bbbg}

\subsection{The harmonic sum as a parameter derivative}
\label{subsec:bbbg-derivative}

Bailey, Borwein, Broadhurst, and Glasser
\cite{MR2450513} obtained a harmonic-series representation of the
four-Bessel moment
\begin{equation}
	s_{4,0}
	:=
	\int_0^\infty I_0(t)K_0(t)^3\,\dd t,
	\label{eq:s40-def}
\end{equation}
where $I_0$ and $K_0$ denote the modified Bessel functions of the first
and second kinds, respectively, of order zero.

Set
\begin{equation}
	c_n
	:=
	\left(\frac{(1/2)_n}{n!}\right)^4
	=
	\frac{\binom{2n}{n}^4}{2^{8n}},
\end{equation}
and
\begin{equation}
	b_n
	:=
	\frac{1}{4n+1}
	+
	2\sum_{k=0}^{2n-1}\frac{1}{2k+1},
\end{equation}
with the inner sum understood to be empty when $n=0$.  Their formula is
\begin{equation}
	s_{4,0}
	=
	\frac{\pi^2}{2}
	\sum_{n=0}^\infty c_nb_n.
	\label{eq:BBBG-series}
\end{equation}

Our first step is to realize the harmonic factor $b_n$ as a parameter
derivative.  We then place the resulting deformation in the same completed
$L$-function and $W(D_5)$ framework used for the complementary elliptic
moments.

Define the very-well-poised deformation
\begin{equation}
	\mathcal W_\tau
	:=
	\F76\!\left(
	\begin{matrix}
		\frac12,\frac54,\frac12,\frac12,\frac12,
		\frac14+\tau,\frac34+\tau\\[1mm]
		\frac14,1,1,1,\frac54-\tau,\frac34-\tau
	\end{matrix};1
	\right).
\end{equation}
If $w_n(\tau)$ denotes its $n$th term, then
\begin{equation}
	w_n(0)=c_n.
\end{equation}
Using the digamma increments $D_a(n)$ introduced above, logarithmic
differentiation gives
\begin{equation}
	\left.
	\frac{w_n'(\tau)}{w_n(\tau)}
	\right|_{\tau=0}
	=
	D_{1/4}(n)
	+
	2D_{3/4}(n)
	+
	D_{5/4}(n)
	=
	4(b_n-1).
\end{equation}
Consequently,
\begin{equation}
	\mathcal W_0'
	=
	4\sum_{n=0}^\infty c_n(b_n-1),
\end{equation}
and hence
\begin{equation}
	\sum_{n=0}^\infty c_nb_n
	=
	\mathcal W_0+\frac14\mathcal W_0'.
	\label{eq:bbbg-parameter-derivative}
\end{equation}
Thus the nested odd-harmonic factor is generated by differentiation in
parameter space.

\subsection{A $W(D_5)$ reflection and the quarter-to-half transformation}
\label{subsec:bbbg-reflection}

Consider the balanced path
\begin{equation}
	x_\tau
	=
	\left(
	\frac14-\tau,\frac14-\tau,\frac14-\tau,\frac14+\tau;
	\frac12;
	\frac34-\tau,\frac34-\tau
	\right).
	\label{eq:bbbg-quarter-path}
\end{equation}
Applying the Bailey--Mishev representation
\eqref{eq:bailey-mishev} gives
\begin{equation}
	L(x_\tau)
	=
	P_x(\tau)\mathcal W_\tau,
	\label{eq:bbbg-quarter-L}
\end{equation}
where
\begin{equation}
	P_x(\tau)
	=
	\frac{\Gamma(3/2)}
	{\pi\Gamma(3/4-\tau)\Gamma(5/4-\tau)
		\Gamma(1/2-2\tau)}.
\end{equation}

Applying Mishev's fundamental involution $\mathcal A$
from \eqref{eq:mishev-involution} sends $x_\tau$ to
\begin{equation}
	y_\tau
	=
	\left(
	\frac14-\tau,\frac14-\tau,\frac12,\frac12-2\tau;
	\frac34-\tau;
	1-2\tau,\frac34-\tau
	\right).
	\label{eq:bbbg-reflected-path}
\end{equation}
By the $W(D_5)$ invariance \eqref{eq:mishev-reflection},
\[
L(x_\tau)=L(y_\tau).
\]

Applying the Bailey--Mishev representation\eqref{eq:bailey-mishev} at $y_\tau$, the three
upper/lower parameter pairs
\[
\frac14-\tau,\qquad
\frac54-\tau,\qquad
\frac34-\tau
\]
cancel.  Thus the reflected ${}_7F_6(1)$ reduces to
\begin{equation}
	F_\tau
	:=
	\F43\!\left(
	\begin{matrix}
		\frac12-2\tau,\frac12-2\tau,\frac12,\frac12\\
		1-2\tau,1-2\tau,1
	\end{matrix};1
	\right).
	\label{eq:bbbg-reflected-F}
\end{equation}
The corresponding prefactor is
\begin{equation}
	P_y(\tau)
	=
	\frac{\Gamma(3/2-2\tau)}
	{\pi\Gamma(3/4-\tau)\Gamma(1/2)
		\Gamma(1-2\tau)^2\Gamma(5/4-\tau)}.
\end{equation}
Consequently,
\begin{equation}
	\mathcal W_\tau
	=
	\mathcal R(\tau)F_\tau,
	\label{eq:bbbg-reflection-factorization}
\end{equation}
where
\begin{equation}
	\begin{aligned}
		\mathcal R(\tau)
		&:=
		\frac{P_y(\tau)}{P_x(\tau)}
		\\
		&=
		\frac{
			\Gamma(3/2-2\tau)\Gamma(1/2-2\tau)
		}{
			\Gamma(3/2)\Gamma(1/2)\Gamma(1-2\tau)^2
		}
		\\
		&=
		(1-4\tau)
		\left[
		\frac{\Gamma(1/2-2\tau)}
		{\sqrt{\pi}\,\Gamma(1-2\tau)}
		\right]^2.
	\end{aligned}
\end{equation}
Thus the $W(D_5)$ reflection converts the quarter-parameter deformation
$\mathcal W_\tau$ into the half-parameter ${}_4F_3(1)$ deformation
$F_\tau$, up to the explicit gamma factor $\mathcal R(\tau)$.

At $\tau=0$,
\begin{equation}
	\mathcal R(0)=1,
	\qquad
	F_0
	=
	\F43\!\left(
	\begin{matrix}
		\frac12,\frac12,\frac12,\frac12\\
		1,1,1
	\end{matrix};1
	\right)
	=
	\sum_{n=0}^{\infty}c_n.
\end{equation}
Moreover,
\begin{equation}
	\mathcal R'(0)=8\log2-4.
\end{equation}
The logarithmic derivative of the $n$th term of $F_\tau$ at
$\tau=0$ is
\begin{equation}
	4\bigl(D_1(n)-D_{1/2}(n)\bigr),
\end{equation}
and hence
\begin{equation}
	F_0'
	=
	4\sum_{n=0}^{\infty}
	c_n\bigl(D_1(n)-D_{1/2}(n)\bigr).
\end{equation}

Differentiating \eqref{eq:bbbg-reflection-factorization} at
$\tau=0$ and using \eqref{eq:bbbg-parameter-derivative} gives
\begin{equation}
	\sum_{n=0}^{\infty}c_nb_n
	=
	2\log2\,F_0
	+
	\sum_{n=0}^{\infty}
	c_n\bigl(D_1(n)-D_{1/2}(n)\bigr).
\end{equation}
Since
\begin{equation}
	D_1(n)-D_{1/2}(n)+2\log2
	=
	\psi(n+1)-\psi\!\left(n+\frac12\right),
\end{equation}
we obtain the quarter-to-half harmonic transformation
\begin{equation}
	\sum_{n=0}^{\infty}c_nb_n
	=
	\sum_{n=0}^{\infty}
	c_n
	\left[
	\psi(n+1)-\psi\!\left(n+\frac12\right)
	\right].
	\label{eq:bbbg-quarter-to-half}
\end{equation}

\subsection{The symmetric supplementary limit}
\label{subsec:bbbg-endgame}

It remains to evaluate the half-parameter harmonic sum on the
right-hand side of \eqref{eq:bbbg-quarter-to-half}.  Consider the
balanced path
\begin{equation}
	z_\eps
	=
	\left(
	\frac12+\frac\eps2,
	\frac12+\frac\eps2,
	\frac12+\frac\eps2,
	\frac12+\frac\eps2;
	1+\eps;
	1+\frac\eps2,
	1+\frac\eps2
	\right).
	\label{eq:bbbg-symmetric-path}
\end{equation}
Define
\begin{equation}
	\Phi(\eps)
	:=
	\F43\!\left(
	\begin{matrix}
		\frac12+\frac\eps2,
		\frac12+\frac\eps2,
		\frac12+\frac\eps2,
		\frac12+\frac\eps2\\
		1+\eps,
		1+\frac\eps2,
		1+\frac\eps2
	\end{matrix};1
	\right),
\end{equation}
and set
\begin{equation}
	\mathcal P(\eps)
	:=
	\Gamma(1+\eps)
	\Gamma\!\left(1+\frac\eps2\right)^2
	\Gamma\!\left(\frac12-\frac\eps2\right)^4,
	\qquad
	\widehat{\Phi}(\eps)
	:=
	\frac{\Phi(\eps)}{\mathcal P(\eps)}.
\end{equation}

Substituting \eqref{eq:bbbg-symmetric-path} into Mishev's completed
$L$-function \eqref{eq:mishev-L}, the supplementary $\F43$ branch is
obtained by replacing $\eps$ with $-\eps$.  Since
\[
\sin\pi(1+\eps)=-\sin\pi\eps,
\]
we obtain
\begin{equation}
	L(z_\eps)
	=
	\frac{
		\widehat{\Phi}(-\eps)-\widehat{\Phi}(\eps)
	}{
		\sin\pi\eps
	}.
	\label{eq:bbbg-symmetric-collision}
\end{equation}
Taking the finite limit as $\eps\to0$ gives
\begin{equation}
	L(z_0)
	=
	-\frac2\pi\,\widehat{\Phi}'(0).
	\label{eq:bbbg-Lz0-derivative}
\end{equation}

At $\eps=0$,
\begin{equation}
	\Phi(0)=\sum_{n=0}^{\infty}c_n,
\end{equation}
and logarithmic differentiation of the $n$th term of $\Phi$ gives
\begin{equation}
	2D_{1/2}(n)-2D_1(n).
\end{equation}
Hence
\begin{equation}
	\Phi'(0)
	=
	-2\sum_{n=0}^{\infty}
	c_n\bigl(D_1(n)-D_{1/2}(n)\bigr).
\end{equation}
Moreover,
\begin{equation}
	\mathcal P(0)=\pi^2,
	\qquad
	\frac{\mathcal P'(0)}{\mathcal P(0)}
	=
	4\log2.
\end{equation}
Therefore
\begin{equation}
	\widehat{\Phi}'(0)
	=
	-\frac2{\pi^2}
	\sum_{n=0}^{\infty}
	c_n
	\left[
	D_1(n)-D_{1/2}(n)+2\log2
	\right],
\end{equation}
and \eqref{eq:bbbg-Lz0-derivative} yields
\begin{equation}
	L(z_0)
	=
	\frac4{\pi^3}
	\sum_{n=0}^{\infty}
	c_n
	\left[
	\psi(n+1)-\psi\!\left(n+\frac12\right)
	\right].
	\label{eq:bbbg-Lz0-half}
\end{equation}

Here
\begin{equation}
	z_0
	=
	\left(
	\frac12,\frac12,\frac12,\frac12;
	1;
	1,1
	\right).
	\label{eq:bbbg-z0}
\end{equation}
Applying the Bailey--Mishev representation
\eqref{eq:bailey-mishev} at $z_0$ gives
\begin{equation}
	L(z_0)=\frac1{2\pi}V,
\end{equation}
where
\begin{equation}
	V
	:=
	\F76\!\left(
	\begin{matrix}
		\frac54,
		\frac12,\frac12,\frac12,\frac12,\frac12,\frac12\\[1mm]
		\frac14,
		1,1,1,1,1
	\end{matrix};1
	\right).
	\label{eq:bbbg-V}
\end{equation}
Comparing with \eqref{eq:bbbg-Lz0-half} gives
\begin{equation}
	\sum_{n=0}^{\infty}
	c_n
	\left[
	\psi(n+1)-\psi\!\left(n+\frac12\right)
	\right]
	=
	\frac{\pi^2}{8}V.
	\label{eq:bbbg-half-evaluation}
\end{equation}
The resulting ${}_7F_6(1)$ evaluation of the Bessel moment is known
through the four-step random-walk formulation of Borwein, Straub, and
Wan \cite[Theorem~4]{MR3038778}, together with the elliptic-integral
representation of $s_{4,0}$ in the work of Bailey, Borwein, Broadhurst,
and Glasser \cite{MR2450513}.  The argument above gives a different
derivation directly from the BBBG harmonic series through Mishev's
completed $L$-function and its $W(D_5)$ symmetry.

Combining \eqref{eq:bbbg-quarter-to-half} with
\eqref{eq:bbbg-half-evaluation} proves the following.


\begin{theorem}[Hypergeometric evaluation of the BBBG four-Bessel moment]

	\label{thm:bbbg}
	With
	\begin{equation}
		c_n
		=
		\frac{\binom{2n}{n}^4}{2^{8n}},
		\qquad
		b_n
		=
		\frac1{4n+1}
		+
		2\sum_{k=0}^{2n-1}\frac1{2k+1},
	\end{equation}
	one has
	\begin{equation}
		\sum_{n=0}^{\infty}c_nb_n
		=
		\frac{\pi^2}{8}
		\F76\!\left(
		\begin{matrix}
			\frac54,\frac12,\frac12,\frac12,\frac12,\frac12,\frac12\\[1mm]
			\frac14,1,1,1,1,1
		\end{matrix};1
		\right).
	\end{equation}
	Consequently,
	\begin{equation}
		s_{4,0}
		=
		\frac{\pi^4}{16}
		\F76\!\left(
		\begin{matrix}
			\frac54,\frac12,\frac12,\frac12,\frac12,\frac12,\frac12\\[1mm]
			\frac14,1,1,1,1,1
		\end{matrix};1
		\right).
	\end{equation}
\end{theorem}

The preceding proof uses the $W(D_5)$ structure in two complementary
ways.  First, Mishev's fundamental involution transports the BBBG
deformation to a cancellation locus, reducing the very-well-poised
${}_7F_6(1)$ derivative to a ${}_4F_3(1)$ problem.  Second, the resulting
half-parameter problem is evaluated at the symmetric point $z_0$ in
\eqref{eq:bbbg-z0}, where the two supplementary ${}_4F_3(1)$ branches
coalesce, while the Bailey--Mishev representation identifies the same
value $L(z_0)$ with the symmetric ${}_7F_6(1)$.  Thus $z_0$ is
simultaneously a collision point for the supplementary representation
and the central symmetric point of the very-well-poised representation.


\begin{remark}[The symmetric point in the $K'^2$ family]
	The same symmetric point also occurs in the complementary $K'^2$ family.
	Indeed, by \eqref{eq:wan-q-path} and \eqref{eq:wan-p-point},
	\[
	\mathfrak q_{1/2}
	=
	\mathfrak{p}_{1/2}
	=
	z_0.
	\]
	Thus the point governing the BBBG supplementary collision is also the
	special member of the $K'^2$ family at which the original parameter point
	and its $W(D_5)$-reflected image coincide.
\end{remark}

Since $z_0$ is also fixed by the $W(D_5)$ action, it is natural to ask
what further information is encoded in the induced linear action on
the tangent space $T_{z_0}\mathcal H$.  This leads to the eigenpath
selection principle developed below.

\section{Symmetry-adapted paths and Taylor selection}
\label{sec:eigenpath}

The preceding arguments use the $W(D_5)$ symmetry globally, through
reflections, collisions, and parameter transformations.  At a point fixed
by the Weyl-group action, the same symmetry has a local manifestation on
the tangent space.  Eigenvectors of the induced linear action determine
distinguished parameter paths along which the Taylor expansion of the
completed $L$-function is constrained by symmetry.

Let $\mathcal H$ denote the balanced affine hyperplane, let
$w\in W(D_5)$ fix a point $x_0\in\mathcal H$, and let
$Dw_{x_0}$ denote the induced linear map on $T_{x_0}\mathcal H$.

\begin{proposition}[Eigenpath selection principle]
	\label{prop:eigenpath}
	Suppose $v\in T_{x_0}\mathcal H$ satisfies
	\[
	Dw_{x_0}(v)=\lambda v,
	\]
	where $\lambda$ is a root of unity.  Let $x(t)=x_0+tv$ and suppose
	$L(x(t))$ is analytic near $t=0$.  Then
	\begin{equation}
		L(x(t))=L(x(\lambda t)).
	\end{equation}
	Consequently, if
	\[
	L(x(t))=\sum_{m\ge0}a_mt^m,
	\]
	then
	\begin{equation}
		(1-\lambda^m)a_m=0.
	\end{equation}
	In particular, $a_m=0$ whenever $\lambda^m\ne1$.
\end{proposition}

\begin{proof}
	Since the action is affine and $w(x_0)=x_0$,
	\[
	w(x_0+tv)
	=
	x_0+t\,Dw_{x_0}(v)
	=
	x_0+\lambda tv
	=
	x(\lambda t).
	\]
	The $W(D_5)$ invariance of $L$ therefore gives
	\[
	L(x(t))
	=
	L(w(x(t)))
	=
	L(x(\lambda t)).
	\]
	Comparison of Taylor coefficients proves the claim.
\end{proof}

Thus an eigenvalue $-1$ forces all odd Taylor coefficients to vanish;
an eigenvalue $i$ permits only degrees divisible by $4$; and a primitive
eighth root permits only degrees divisible by $8$.
\subsection{A genuine Coxeter element at the symmetric point}

At the symmetric point $z_0$ defined in \eqref{eq:bbbg-z0}, consider
Mishev's fundamental involution $\mathcal A$ together with the parameter
permutations
\begin{equation}
	\cC:=(123)(67)\mathcal A.
	\label{eq:coxeter-element}
\end{equation}
In coordinates, one convenient realization is
\begin{equation}
	\cC(a,b,c,d;e;f,g)
	=
	(g-c,a,b,g-d;1+a+b-f;g,1+a+b-e).
	\label{eq:coxeter-action}
\end{equation}
The point $z_0$ is fixed by $\cC$.  Since $\cC$ is affine, its
differential
\[
D\cC_{z_0}:T_{z_0}\mathcal H\longrightarrow T_{z_0}\mathcal H
\]
is the induced linear action on the six-dimensional balanced tangent
space.  Its characteristic polynomial is
\begin{equation}
	(\lambda-1)(\lambda+1)(\lambda^4+1).
	\label{eq:coxeter-charpoly}
\end{equation}
Thus primitive eighth roots occur naturally, as expected for a Coxeter
element of type $D_5$, whose Coxeter number is $8$.

Let
\begin{equation}
	\omega=e^{\pi i/4}.
\end{equation}
A corresponding $\omega$-eigenvector is
\begin{equation}
	v_\omega
	=
	\left(
	\omega^2,\omega,1,1-\omega+\omega^2;
	1+\omega^2;
	\omega^2-\omega^3,1+\omega^3
	\right).
	\label{eq:coxeter-eigenvector}
\end{equation}
A direct substitution gives
\begin{equation}
	D\cC_{z_0}(v_\omega)=\omega v_\omega.
	\label{eq:coxeter-eigenvalue}
\end{equation}

Define
\begin{equation}
	x_\omega(t):=z_0+t v_\omega,
	\qquad
	\cL_\omega(t):=L(x_\omega(t)).
\end{equation}
Since $\cC$ is affine and fixes $z_0$,
\begin{equation}
	\cC x_\omega(t)=x_\omega(\omega t),
	\qquad
	\cL_\omega(t)=\cL_\omega(\omega t).
\end{equation}
By Proposition~\ref{prop:eigenpath},
\begin{equation}
	\cL_\omega^{(m)}(0)=0,
	\qquad
	(m\ge1,\;8\nmid m).
	\label{eq:coxeter-vanishing}
\end{equation}
These symmetry-forced vanishings are the invariant-theoretic source of
the identities developed below.
\subsection{The $W(D_5)$-fixed line through the BBBG point}
\label{sec:D5-fixed-line}

The Coxeter construction above is centered at the symmetric point
$z_0$.  The same symmetry persists along a one-parameter family.  Indeed,
the $D_5$-coordinate description shows that the full Weyl group fixes the
affine line
\begin{equation}
	z_s=(s,s,s,s;1;2s,2s).
	\label{eq:D5-fixed-line}
\end{equation}
In particular,
\[
z_{1/2}=z_0.
\]

At $z_s$, the Bailey--Mishev representation becomes
\begin{equation}
	{}_7F_6\!\left(
	\begin{matrix}
		3s-1,\frac{3s+1}{2},s,s,s,s,s\\
		\frac{3s-1}{2},2s,2s,2s,2s,2s
	\end{matrix};1
	\right)
	=
	\sum_{n=0}^{\infty}q_n(s),
	\label{eq:D5-fixed-series}
\end{equation}
where
\begin{equation}
	q_n(s)
	:=
	\frac{2n+3s-1}{3s-1}
	\frac{(3s-1)_n(s)_n^5}{(2s)_n^5\,n!}.
	\label{eq:D5-fixed-qn}
\end{equation}

Because the $W(D_5)$ action is affine, its linear part is independent
of the base point.  Hence, under the natural identification of the tangent
spaces of $\mathcal H$,
\[
D\cC_{z_s}=D\cC_{z_0}.
\]
Therefore the Coxeter eigenvector $v_\omega$ found at $z_0$ remains an
$\omega$-eigenvector along the entire fixed line.
Applying Proposition~\ref{prop:eigenpath} to the first Taylor coefficient
therefore gives the one-parameter identity
\begin{equation}
	\sum_{n=0}^{\infty}q_n(s)G_n(s)=0,
	\label{eq:D5-fixed-first-identity}
\end{equation}
where
\begin{equation}
	G_n(s)
	=
	\psi(n+3s-1)
	+3\psi(n+s)
	-2\psi(n+2s)
	-2\psi(s)
	+\frac{1}{2n+3s-1}.
	\label{eq:D5-fixed-Gn}
\end{equation}

At $s=\frac12$,
\begin{equation}
	q_n\!\left(\frac12\right)
	=
	(4n+1)
	\left(\frac{(1/2)_n}{n!}\right)^6,
\end{equation}
so \eqref{eq:D5-fixed-first-identity} reduces to the first BBBG-type
harmonic identity at the symmetric point.  At $s=1$,
\begin{equation}
	q_n(1)=\frac{1}{(n+1)^3},
	\label{eq:D5-fixed-s1-weight}
\end{equation}
and \eqref{eq:D5-fixed-first-identity} becomes Euler's classical evaluation
\begin{equation}
	\sum_{k=1}^{\infty}\frac{H_k}{k^3}
	=
	\frac54\zeta(4).
	\label{eq:Euler-Hk-k3}
\end{equation}
Higher Taylor coefficients yield further one-parameter identities along
the fixed line, but their harmonic structure rapidly becomes more
complicated.  We therefore restrict here to the first-order relation,
which already interpolates between the BBBG central-binomial weight at
$s=\frac12$ and the classical Euler weight at $s=1$.
\begin{remark}
Further symmetry-enhanced points arising
	from the mixed $E'K'$ and $E'^2$ moment families also admit finite-order
	eigenpaths and corresponding Taylor-selection identities.  Since these
	additional identities are not needed for the main development, we do not
	record them here.
\end{remark}

\subsection{Arithmetic consequences of the Coxeter eigenpath}
\label{sec:arithmetic-bbbg}
We now extract arithmetic consequences of the symmetry-forced Taylor
vanishing along the Coxeter eigenpath.  We work out the first three orders
in detail, both to exhibit the resulting harmonic sums and to make the
general mechanism transparent.
Recall the eigenpath
\[
x_\omega(t)=z_0+t v_\omega,
\qquad
\cL_\omega(t):=L(x_\omega(t)).
\]
By Proposition~\ref{prop:eigenpath},
\begin{equation}
	\cL_\omega^{(m)}(0)=0
	\qquad
	(m\ge1,\;8\nmid m).
	\label{eq:arithmetic-coxeter-vanishing}
\end{equation}

At $t=0$, the hypergeometric factor is the symmetric series $V$
defined in \eqref{eq:bbbg-V}.  Specializing the fixed-line weight
\eqref{eq:D5-fixed-qn} at $s=\frac12$, write
\begin{equation}
	q_n
	:=
	q_n\!\left(\frac12\right)
	=
	(4n+1)
	\left(\frac{(1/2)_n}{n!}\right)^6
	=
	(4n+1)\frac{\binom{2n}{n}^6}{2^{12n}}.
	\label{eq:bbbg-qn}
\end{equation}
Then
\begin{equation}
	V=\sum_{n=0}^{\infty}q_n.
	\label{eq:bbbg-V-series}
\end{equation}

Let $T_n(t)$ denote the $n$th hypergeometric term in the
Bailey--Mishev representation of $\cL_\omega(t)$, and let
$P_\omega(t)$ denote the corresponding gamma prefactor.  Thus
\begin{equation}
	\cL_\omega(t)
	=
	P_\omega(t)\sum_{n=0}^{\infty}T_n(t),
	\qquad
	T_n(0)=q_n.
	\label{eq:bbbg-eigenpath-series}
\end{equation}
It is convenient to normalize the completed summand by
\begin{equation}
	h_n(t)
	:=
	\frac{P_\omega(t)T_n(t)}
	{P_\omega(0)q_n},
	\label{eq:bbbg-hn-def}
\end{equation}
so that $h_n(0)=1$ and
\begin{equation}
	\cL_\omega(t)
	=
	P_\omega(0)
	\sum_{n=0}^{\infty}q_nh_n(t).
	\label{eq:bbbg-normalized-series}
\end{equation}
Write
\begin{equation}
	\ell_n(t):=\log h_n(t).
	\label{eq:bbbg-elln-def}
\end{equation}

The advantage of passing to $\ell_n$ is that logarithmic
differentiation separates the gamma factors from the Pochhammer
factors.  If a parameter varies linearly as
\[
u(t)=u+\sigma t,
\]
then
\begin{equation}
	\left.
	\frac{d}{dt}\log\Gamma(u(t))
	\right|_{t=0}
	=
	\sigma\psi(u),
\end{equation}
whereas
\begin{equation}
	\left.
	\frac{d}{dt}\log (u(t))_n
	\right|_{t=0}
	=
	\sigma D_u(n),
\end{equation}
where $D_u(n)=\psi(u+n)-\psi(u)$ is the digamma increment introduced
earlier.  Thus gamma factors contribute ordinary digamma values, while
Pochhammer factors contribute finite harmonic sums.

\subsubsection{The first vanishing Taylor coefficient}

Set
\begin{equation}
	\rho:=1-\sqrt2.
\end{equation}
Since
\begin{equation}
	\ell_n(t)
	=
	\log\frac{P_\omega(t)}{P_\omega(0)}
	+
	\log\frac{T_n(t)}{T_n(0)},
\end{equation}
we have
\begin{equation}
	\ell_n'(0)
	=
	\left.
	\frac{d}{dt}\log P_\omega(t)
	\right|_{t=0}
	+
	\left.
	\frac{d}{dt}\log T_n(t)
	\right|_{t=0}.
	\label{eq:ell-first-split}
\end{equation}

We first consider the hypergeometric term.  Write
\begin{equation}
	v_\omega
	=
	(v_a,v_b,v_c,v_d;v_e;v_f,v_g),
\end{equation}
so that the components of $v_\omega$ are the slopes of the original
parameters $a,\ldots,g$ along the eigenpath
$x_\omega(t)=z_0+t v_\omega$.

Each Pochhammer parameter in the Bailey--Mishev ${}_7F_6$ is an
affine-linear combination of these original parameters.  Thus, if
\[
u(t)=u+\sigma t,
\]
then
\begin{equation}
	\left.
	\frac{d}{dt}\log (u(t))_n
	\right|_{t=0}
	=
	\sigma D_u(n),
\end{equation}
where $D_u(n)$ is the digamma increment introduced in
\eqref{eq:digamma-increment}.  Hence the value $u=u(0)$ determines the harmonic
increment, while the slope $\sigma$ determines its coefficient.

Using the parameterization in \eqref{eq:bailey-mishev}, recall that
\begin{equation}
	\alpha=d+g-e.
\end{equation}
Along the eigenpath,
\begin{equation}
	\alpha(0)=\frac12,
	\qquad
	\alpha'(0)
	=
	v_d+v_g-v_e
	=
	1-\omega+\omega^3
	=
	\rho.
\end{equation}
Similarly, $g-a$ has base value $1/2$ and slope $v_g-v_a$, while
$1+d-e$ has base value $1/2$ and slope $v_d-v_e$.

Grouping the Pochhammer parameters according to their values at $t=0$,
the six numerator parameters based at $1/2$ have total slope $4\rho$,
while the five denominator parameters based at $1$ have total slope
$2\rho$.  The remaining numerator and denominator parameters,
$1+\alpha/2$ and $\alpha/2$, are based at $5/4$ and $1/4$ and both
have slope $\rho/2$.  Therefore
\begin{equation}
	\left.
	\frac{d}{dt}\log T_n(t)
	\right|_{t=0}
	=
	4\rho D_{1/2}(n)
	-
	2\rho D_1(n)
	+
	\frac{\rho}{2}
	\bigl(
	D_{5/4}(n)-D_{1/4}(n)
	\bigr).
\end{equation}
Using
\begin{equation}
	D_{5/4}(n)-D_{1/4}(n)
	=
	-4+\frac{4}{4n+1},
\end{equation}
we obtain
\begin{equation}
	\left.
	\frac{d}{dt}\log T_n(t)
	\right|_{t=0}
	=
	\rho
	\left(
	4D_{1/2}(n)
	-
	2D_1(n)
	-
	2
	+
	\frac{2}{4n+1}
	\right).
	\label{eq:T-first-derivative}
\end{equation}

The gamma prefactor is treated in the same way.  At $t=0$, its numerator
gamma factor has argument $3/2$ and slope $\rho$.  Among the denominator
factors, one has argument $1/2$ and slope $-\rho$, while the remaining
five have argument $1$ and total slope $2\rho$.  Hence
\begin{equation}
	\left.
	\frac{d}{dt}\log P_\omega(t)
	\right|_{t=0}
	=
	\rho
	\left[
	\psi\!\left(\frac32\right)
	+
	\psi\!\left(\frac12\right)
	-
	2\psi(1)
	\right].
\end{equation}
Using
\[
\psi\!\left(\frac32\right)
=
\psi\!\left(\frac12\right)+2,
\qquad
\psi(1)-\psi\!\left(\frac12\right)=2\log2,
\]
this simplifies to
\begin{equation}
	\left.
	\frac{d}{dt}\log P_\omega(t)
	\right|_{t=0}
	=
	\rho(2-4\log2).
	\label{eq:P-first-derivative}
\end{equation}

Adding \eqref{eq:T-first-derivative} and
\eqref{eq:P-first-derivative}, the constant terms cancel and give
\begin{equation}
	\ell_n'(0)
	=
	2\rho
	\left(
	2D_{1/2}(n)
	-
	D_1(n)
	-
	2\log2
	+
	\frac{1}{4n+1}
	\right).
\end{equation}
Since
\[
D_1(n)=H_n,
\qquad
D_{1/2}(n)=2H_{2n}-H_n,
\]
we obtain
\begin{equation}
	\ell_n'(0)=2\rho A_n,
\end{equation}
where
\begin{equation}
	A_n
	:=
	4H_{2n}
	-
	3H_n
	-
	2\log2
	+
	\frac{1}{4n+1}.
	\label{eq:def-An}
\end{equation}

Since $h_n(0)=1$,
\begin{equation}
	h_n'(0)=\ell_n'(0)=2\rho A_n.
\end{equation}
Differentiating \eqref{eq:bbbg-normalized-series} at $t=0$ gives
\begin{equation}
	\cL_\omega'(0)
	=
	2\rho\,P_\omega(0)
	\sum_{n=0}^{\infty}q_nA_n.
\end{equation}
By \eqref{eq:arithmetic-coxeter-vanishing},
\[
\cL_\omega'(0)=0.
\]
Since $\rho\ne0$ and $P_\omega(0)\ne0$, we obtain the first harmonic
identity.

\begin{proposition}
	\label{prop:first-harmonic-identity}
	With $q_n$ defined in \eqref{eq:bbbg-qn} and $A_n$ defined in
	\eqref{eq:def-An}, one has
	\begin{equation}
		\sum_{n=0}^{\infty}q_nA_n=0.
	\end{equation}
\end{proposition}

This is the simplest arithmetic consequence of the Coxeter eigenpath.
The harmonic factor arises directly from logarithmic differentiation of
the completed hypergeometric summand.
\subsubsection{The second vanishing Taylor coefficient}

The second derivative illustrates the next feature of the method:
although logarithmic derivatives remain linear in generalized harmonic
numbers, derivatives of $h_n=e^{\ell_n}$ produce nonlinear
combinations of them.

Define
\begin{equation}
	D_a^{(2)}(n)
	:=
	\sum_{j=0}^{n-1}\frac{1}{(a+j)^2}.
	\label{eq:second-digamma-increment}
\end{equation}
If
\[
u(t)=u+\sigma t,
\]
then
\begin{equation}
	\left.
	\frac{d^2}{dt^2}\log (u(t))_n
	\right|_{t=0}
	=
	-\sigma^2D_u^{(2)}(n).
\end{equation}
A denominator Pochhammer factor contributes with the opposite sign.

Using the same parameter slopes as above, the second logarithmic
derivative of the hypergeometric term is
\begin{equation}
	\left.
	\frac{d^2}{dt^2}\log T_n(t)
	\right|_{t=0}
	=
	-2\rho^2D_{1/2}^{(2)}(n)
	+
	\frac{\rho^2}{4}
	\left(
	D_{1/4}^{(2)}(n)-D_{5/4}^{(2)}(n)
	\right).
\end{equation}
The shift relation
\begin{equation}
	D_{5/4}^{(2)}(n)
	=
	D_{1/4}^{(2)}(n)
	-
	16
	+
	\frac{16}{(4n+1)^2}
\end{equation}
therefore gives
\begin{equation}
	\left.
	\frac{d^2}{dt^2}\log T_n(t)
	\right|_{t=0}
	=
	-2\rho^2D_{1/2}^{(2)}(n)
	+
	4\rho^2
	-
	\frac{4\rho^2}{(4n+1)^2}.
	\label{eq:T-second-derivative}
\end{equation}

For the gamma prefactor, the same slope calculation gives
\begin{equation}
	\left.
	\frac{d^2}{dt^2}\log P_\omega(t)
	\right|_{t=0}
	=
	\rho^2
	\left[
	\psi^{(1)}\!\left(\frac32\right)
	-
	\psi^{(1)}\!\left(\frac12\right)
	\right].
\end{equation}
Using the trigamma recurrence
\[
\psi^{(1)}\!\left(\frac32\right)
=
\psi^{(1)}\!\left(\frac12\right)-4,
\]
we obtain
\begin{equation}
	\left.
	\frac{d^2}{dt^2}\log P_\omega(t)
	\right|_{t=0}
	=
	-4\rho^2.
	\label{eq:P-second-derivative}
\end{equation}

Adding \eqref{eq:T-second-derivative} and
\eqref{eq:P-second-derivative}, the constant terms again cancel:
\begin{equation}
	\ell_n''(0)
	=
	-2\rho^2
	\left(
	D_{1/2}^{(2)}(n)
	+
	\frac{2}{(4n+1)^2}
	\right).
\end{equation}

For $r\ge1$, write
\begin{equation}
	H_n^{(r)}
	:=
	\sum_{k=1}^n\frac{1}{k^r}.
	\label{eq:generalized-harmonic}
\end{equation}
Since
\begin{equation}
	D_{1/2}^{(2)}(n)
	=
	4H_{2n}^{(2)}-H_n^{(2)},
\end{equation}
we obtain
\begin{equation}
	\ell_n''(0)
	=
	-2\rho^2R_n,
\end{equation}
where
\begin{equation}
	R_n
	:=
	4H_{2n}^{(2)}
	-
	H_n^{(2)}
	+
	\frac{2}{(4n+1)^2}.
	\label{eq:def-Bn}
\end{equation}

We now return from $\ell_n$ to $h_n=e^{\ell_n}$.  Since
\begin{equation}
	h_n''(0)
	=
	\ell_n'(0)^2+\ell_n''(0),
\end{equation}
and $\ell_n'(0)=2\rho A_n$, we find
\begin{equation}
	h_n''(0)
	=
	4\rho^2
	\left(
	A_n^2-\frac12R_n
	\right).
\end{equation}
Define
\begin{equation}
	X_n
	:=
	A_n^2-\frac12R_n.
	\label{eq:def-Xn}
\end{equation}
Thus
\begin{equation}
	h_n''(0)=4\rho^2X_n.
\end{equation}

Differentiating \eqref{eq:bbbg-normalized-series} twice at $t=0$ gives
\begin{equation}
	\cL_\omega''(0)
	=
	4\rho^2P_\omega(0)
	\sum_{n=0}^{\infty}q_nX_n.
\end{equation}
By \eqref{eq:arithmetic-coxeter-vanishing},
\[
\cL_\omega''(0)=0.
\]
Since $\rho\ne0$ and $P_\omega(0)\ne0$, we obtain the second harmonic
identity.

\begin{proposition}
	\label{prop:second-harmonic-identity}
	With $q_n$, $A_n$, $R_n$, and $X_n$ defined in
	\eqref{eq:bbbg-qn}, \eqref{eq:def-An}, \eqref{eq:def-Bn}, and
	\eqref{eq:def-Xn}, respectively, one has
	\begin{equation}
		\sum_{n=0}^{\infty}q_nX_n=0.
	\end{equation}
\end{proposition}

The first two orders already exhibit the general structure.  The
logarithmic derivatives $\ell_n^{(r)}(0)$ are linear combinations of
generalized harmonic sums, whereas the derivatives of
$h_n=e^{\ell_n}$ are assembled from them through exponential Bell
polynomials.  The Coxeter selection rule then converts the vanishing
Taylor coefficients of $\cL_\omega$ into harmonic sum identities.

\subsubsection{The third-order projective consequence}

Since $8\nmid3$, the original primitive-$\omega$ eigenpath already yields
a valid third-order identity.  That identity, however, does not isolate the
genuinely new cubic harmonic contribution.  Indeed,
\begin{equation}
	\frac{h_n'''(0)}{h_n(0)}
	=
	\ell_n'(0)^3
	+
	3\ell_n'(0)\ell_n''(0)
	+
	\ell_n'''(0),
\end{equation}
and, since
\[
\ell_n'(0)=2\rho A_n,
\qquad
\ell_n''(0)=-2\rho^2R_n,
\]
the resulting relation contains mixed lower-order terms such as
$A_n^3$ and $A_nR_n$.

To separate the new third-order contribution, we use the larger
$i$-eigenspace of $\cC^2$.  The $\omega$- and $-\omega$-eigenspaces
of $\cC$ both become $i$-eigenspaces for $\cC^2$, since
\begin{equation}
	\omega^2=(-\omega)^2=i.
\end{equation}
Consequently, the $i$-eigenspace of $(D\cC_{z_0})^2$ is
two-dimensional.  A convenient basis is
\begin{equation}
	v_0
	=
	\left(
	\frac{1+i}{2},
	\frac{1-i}{2},
	\frac{1-i}{2},
	\frac{1+i}{2};
	1;0,1
	\right),
\end{equation}
\begin{equation}
	v_1
	=
	\left(
	\frac{1+i}{2},
	\frac{1+i}{2},
	\frac{1-i}{2},
	\frac{1-i}{2};
	1;1,0
	\right),
\end{equation}
for which
\begin{equation}
	(D\cC_{z_0})^2v_0=iv_0,
	\qquad
	(D\cC_{z_0})^2v_1=iv_1.
\end{equation}
Thus
\begin{equation}
	E_i
	:=
	\ker\!\bigl((D\cC_{z_0})^2-iI\bigr)
	=
	\operatorname{span}_{\mathbb C}\{v_0,v_1\},
	\qquad
	v(p):=v_0+p\,v_1.
	\label{eq:Ei-projective-family}
\end{equation}

Let $\mathcal S_n(x)$ denote the $n$th completed summand in the
Bailey--Mishev representation of $L(x)$ near $z_0$.  At the symmetric
point, the first logarithmic directional derivative restricts to a
single linear functional on $E_i$.  More precisely, there is a linear
form
\begin{equation}
	\varphi:E_i\longrightarrow\mathbb C
\end{equation}
such that
\begin{equation}
	D_v\log\mathcal S_n(z_0)
	=
	\varphi(v)A_n,
	\qquad
	v\in E_i.
	\label{eq:projective-first-factorization}
\end{equation}
On the chosen basis,
\begin{equation}
	\varphi(v_0)=1+i,
	\qquad
	\varphi(v_1)=-(1+i),
\end{equation}
and hence
\begin{equation}
	\varphi(v(p))
	=
	(1+i)(1-p).
\end{equation}

The exceptional direction $p=1$ is therefore precisely the kernel of
the first logarithmic derivative.  For $p\ne1$, normalize the eigenline
by setting
\begin{equation}
	\widehat v(p)
	:=
	\frac{v(p)}{(1+i)(1-p)},
	\label{eq:projective-normalization}
\end{equation}
so that
\begin{equation}
	\varphi(\widehat v(p))=1.
\end{equation}
Define the corresponding affine path
\begin{equation}
	x_p(t):=z_0+t\,\widehat v(p).
	\label{eq:projective-path}
\end{equation}
Then
\begin{equation}
	x_p(0)=z_0,
	\qquad
	x_p'(0)=\widehat v(p),
\end{equation}
and
\begin{equation}
	D_{\widehat v(p)}
	\log\mathcal S_n(z_0)
	=
	A_n,
\end{equation}
independently of the choice of eigenline.

There is a further low-order degeneracy.  The second logarithmic
directional derivative defines a quadratic form
\begin{equation}
	Q_n(v)
	:=
	D_v^2\log\mathcal S_n(z_0),
	\qquad
	v\in E_i,
\end{equation}
and the explicit parameter calculation gives
\begin{equation}
	Q_n(v)
	=
	-\frac12R_n\,\varphi(v)^2.
	\label{eq:projective-second-factorization}
\end{equation}
Thus the Hessian has rank one on $E_i$ and annihilates the same
projective direction as the first derivative.  After the normalization
\eqref{eq:projective-normalization},
\begin{equation}
	D_{\widehat v(p)}^2
	\log\mathcal S_n(z_0)
	=
	-\frac12R_n,
\end{equation}
again independently of $p$.

Hence the first and second logarithmic derivatives are projectively
rigid; genuinely new dependence on the eigenline first appears at
third order.

Since $\widehat v(p)$ is an $i$-eigenvector of $(D\cC_{z_0})^2$,
the affine action gives
\begin{equation}
	\cC^2x_p(t)=x_p(it),
	\qquad
	L(x_p(t))=L(x_p(it)).
	\label{eq:projective-selection}
\end{equation}
By Proposition~\ref{prop:eigenpath},
\begin{equation}
	\left.
	\frac{\dd^m}{\dd t^m}L(x_p(t))
	\right|_{t=0}
	=
	0,
	\qquad
	m\not\equiv0\pmod4.
	\label{eq:projective-vanishing}
\end{equation}
In particular, the first three derivatives vanish for every admissible
projective direction $p$.

To measure the remaining projective dependence, set
\begin{equation}
	\xi
	:=
	\left(\frac{p+1}{p-1}\right)^2.
	\label{eq:projective-coordinate}
\end{equation}
Thus $p=-1$ corresponds to $\xi=0$, while the null direction $p=1$
is sent to infinity.

Define normalized completed summands along the path $x_p(t)$ by
\begin{equation}
	h_{n,p}(t)
	:=
	2\pi\,\mathcal S_n(x_p(t)).
\end{equation}
Then
\begin{equation}
	2\pi L(x_p(t))
	=
	\sum_{n=0}^{\infty}h_{n,p}(t),
	\qquad
	h_{n,p}(0)=q_n.
	\label{eq:projective-completed-series}
\end{equation}
The constant factor $2\pi$ does not affect positive-order logarithmic
derivatives, so write
\begin{equation}
	\lambda_{r,n}(p)
	:=
	\left.
	\frac{\dd^r}{\dd t^r}
	\log h_{n,p}(t)
	\right|_{t=0}.
\end{equation}
The preceding factorizations give
\begin{equation}
	\lambda_{1,n}(p)=A_n,
	\qquad
	\lambda_{2,n}(p)=-\frac12R_n,
	\label{eq:projective-first-two}
\end{equation}
independently of $p$.

To compute the third logarithmic derivative, define
\begin{equation}
	D_a^{(3)}(n)
	:=
	\sum_{k=0}^{n-1}\frac{1}{(a+k)^3}.
	\label{eq:third-digamma-increment}
\end{equation}
For an affine Pochhammer parameter $a+\sigma t$,
\begin{equation}
	\left.
	\frac{\dd^3}{\dd t^3}
	\log(a+\sigma t)_n
	\right|_{t=0}
	=
	2\sigma^3D_a^{(3)}(n),
\end{equation}
with the opposite sign for a denominator parameter.

Along $x_p(t)$, let $\mu_j$ denote the six slopes of the numerator
parameters based at $1/2$, and let $\nu_j$ denote the five slopes of
the denominator parameters based at $1$.  Direct substitution of
$\widehat v(p)$ gives
\begin{equation}
	\sum_j\mu_j=2,
	\qquad
	\sum_j\nu_j=1,
\end{equation}
\begin{equation}
	\sum_j\mu_j^2=\frac12,
	\qquad
	\sum_j\nu_j^2=0,
\end{equation}
and, crucially,
\begin{equation}
	\sum_j\mu_j^3
	=
	\frac{1+3\xi}{8},
	\qquad
	\sum_j\nu_j^3
	=
	-\frac{1+3\xi}{8}.
	\label{eq:projective-cubic-slopes}
\end{equation}
The exceptional pair
\[
1+\frac{\alpha}{2},
\qquad
\frac{\alpha}{2},
\qquad
\alpha=d+g-e,
\]
has common slope $1/4$.  Therefore the hypergeometric contribution is
\begin{equation}
	\left.
	\frac{\dd^3}{\dd t^3}
	\log T_{n,p}(t)
	\right|_{t=0}
	=
	\frac{1+3\xi}{4}
	\left(
	D_{1/2}^{(3)}(n)+D_1^{(3)}(n)
	\right)
	+
	\frac1{32}
	\left(
	D_{5/4}^{(3)}(n)-D_{1/4}^{(3)}(n)
	\right).
\end{equation}
Now
\begin{equation}
	D_{1/2}^{(3)}(n)
	=
	8H_{2n}^{(3)}-H_n^{(3)},
	\qquad
	D_1^{(3)}(n)=H_n^{(3)},
\end{equation}
while
\begin{equation}
	D_{5/4}^{(3)}(n)-D_{1/4}^{(3)}(n)
	=
	64
	\left(
	\frac1{(4n+1)^3}-1
	\right).
\end{equation}
Hence
\begin{equation}
	\left.
	\frac{\dd^3}{\dd t^3}
	\log T_{n,p}(t)
	\right|_{t=0}
	=
	2(1+3\xi)H_{2n}^{(3)}
	+
	\frac{2}{(4n+1)^3}
	-
	2.
	\label{eq:T-third-derivative}
\end{equation}

The gamma prefactor contributes
\begin{equation}
	\left.
	\frac{\dd^3}{\dd t^3}
	\log P_p(t)
	\right|_{t=0}
	=
	\frac18
	\left[
	\psi^{(2)}\!\left(\frac32\right)
	+
	\psi^{(2)}\!\left(\frac12\right)
	\right]
	+
	\frac{1+3\xi}{8}\psi^{(2)}(1).
\end{equation}
The classical series representation
\begin{equation}
	\psi^{(2)}(z)
	=
	-2\sum_{m=0}^{\infty}\frac{1}{(m+z)^3}
	=
	-2\zeta(3,z)
\end{equation}
gives
\begin{equation}
	\psi^{(2)}(1)=-2\zeta(3),
	\qquad
	\psi^{(2)}\!\left(\frac12\right)=-14\zeta(3),
	\qquad
	\psi^{(2)}\!\left(\frac32\right)=16-14\zeta(3).
\end{equation}
Therefore
\begin{equation}
	\left.
	\frac{\dd^3}{\dd t^3}
	\log P_p(t)
	\right|_{t=0}
	=
	2
	-
	\frac{15}{4}\zeta(3)
	-
	\frac34\xi\,\zeta(3).
	\label{eq:P-third-derivative}
\end{equation}

Adding \eqref{eq:T-third-derivative} and
\eqref{eq:P-third-derivative}, the constants $-2$ and $2$ cancel and
we obtain
\begin{equation}
	\lambda_{3,n}(p)
	=
	2H_{2n}^{(3)}
	+
	\frac{2}{(4n+1)^3}
	-
	\frac{15}{4}\zeta(3)
	+
	\frac34\xi
	\left(
	8H_{2n}^{(3)}-\zeta(3)
	\right).
	\label{eq:lambda3-projective}
\end{equation}
Writing $[\xi]$ for the coefficient of $\xi$, this gives
\begin{equation}
	[\xi]\,\lambda_{3,n}(p)
	=
	\frac34
	\left(
	8H_{2n}^{(3)}-\zeta(3)
	\right).
	\label{eq:lambda3-projective-coeff}
\end{equation}

Returning from logarithmic derivatives to the completed summand,
\begin{equation}
	\frac{h_{n,p}'''(0)}{h_{n,p}(0)}
	=
	\lambda_{1,n}^3
	+
	3\lambda_{1,n}\lambda_{2,n}
	+
	\lambda_{3,n}.
\end{equation}
Since $\lambda_{1,n}$ and $\lambda_{2,n}$ are independent of $\xi$,
the coefficient of $\xi$ comes entirely from $\lambda_{3,n}$.  Hence
\begin{equation}
	[\xi]\,
	\frac{h_{n,p}'''(0)}{h_{n,p}(0)}
	=
	\frac34
	\left(
	8H_{2n}^{(3)}-\zeta(3)
	\right).
\end{equation}

By \eqref{eq:projective-vanishing}, the third derivative of
$L(x_p(t))$ vanishes for every $p$.  Since the resulting expression is
affine in $\xi$, its $\xi$-coefficient must vanish.  Therefore
\begin{equation}
	\sum_{n=0}^{\infty}
	q_n
	\left(
	8H_{2n}^{(3)}-\zeta(3)
	\right)
	=
	0.
	\label{eq:third-order-projective-identity}
\end{equation}
Equivalently,
\begin{equation}
	\sum_{n=0}^{\infty}
	q_nH_{2n}^{(3)}
	=
	\frac{\zeta(3)}8
	\sum_{n=0}^{\infty}q_n.
	\label{eq:third-order-harmonic}
\end{equation}
Since $\sum_{n\ge0}q_n=V$ and Theorem~\ref{thm:bbbg} gives
$V=16s_{4,0}/\pi^4$, this may also be written as
\begin{equation}
	\sum_{n=0}^{\infty}
	(4n+1)
	\frac{\binom{2n}{n}^6}{2^{12n}}
	H_{2n}^{(3)}
	=
	\frac{2\zeta(3)}{\pi^4}s_{4,0}.
	\label{eq:third-order-bessel}
\end{equation}

\begin{remark}[Companion third-order identity]
	The constant term in $\xi$ yields
	\begin{equation}
		\sum_{n=0}^{\infty}q_n
		\left[
		A_n^3
		-\frac32A_nR_n
		+2H_{2n}^{(3)}
		+\frac{2}{(4n+1)^3}
		-\frac{15}{4}\zeta(3)
		\right]
		=
		0.
	\end{equation}
	Using \eqref{eq:third-order-harmonic}, this simplifies to
	\begin{equation}
		\sum_{n=0}^{\infty}q_n
		\left[
		A_n^3
		-\frac32A_nR_n
		+\frac{2}{(4n+1)^3}
		\right]
		=
		\frac72\,\zeta(3)V.
	\end{equation}
\end{remark}

Thus the change of eigenpath is not needed to obtain a third-order
identity: the original primitive-$\omega$ path already produces one.
The advantage of passing to the two-dimensional $i$-eigenspace of
$\cC^2$ is that the additional projective freedom separates the
genuinely new cubic harmonic contribution from the mixed lower-order
terms.  This projective coefficient-extraction mechanism will reappear
in the higher-order identities below.
\section{Higher-order harmonic identities}
\label{sec:higher-order}

We continue with the normalized projective family in the two-dimensional
$i$-eigenspace $E_i$ constructed in the preceding subsection.  Recall the
projective coordinate
\[
\xi
=
\left(\frac{p+1}{p-1}\right)^2
\]
and the logarithmic derivatives
\[
\lambda_{r,n}(p)
=
\left.
\frac{\dd^r}{\dd t^r}\log h_{n,p}(t)
\right|_{t=0}.
\]

The highest projective-degree terms needed below are governed by the
combinations
\begin{equation}
	U_{3,n}
	:=
	8H_{2n}^{(3)}-\zeta(3),
\end{equation}
\begin{equation}
	U_{4,n}
	:=
	16H_{2n}^{(4)}-2H_n^{(4)}+\zeta(4),
\end{equation}
\begin{equation}
	U_{5,n}
	:=
	32H_{2n}^{(5)}-2H_n^{(5)}+\zeta(5),
\end{equation}
and
\begin{equation}
	U_{7,n}
	:=
	128H_{2n}^{(7)}-\zeta(7).
\end{equation}
More precisely, suppressing the argument $p$ for readability, the top
projective pieces are
\begin{equation}
	\lambda_{3,n}
	=
	\frac34U_{3,n}\xi
	+
	\text{lower powers of $\xi$},
\end{equation}
\begin{equation}
	\lambda_{4,n}
	=
	\frac38U_{4,n}\xi^2
	+
	\text{lower powers of $\xi$},
\end{equation}
\begin{equation}
	\lambda_{5,n}
	=
	-\frac{15}{8}U_{5,n}\xi^2
	+
	\text{lower powers of $\xi$},
\end{equation}
and
\begin{equation}
	\lambda_{7,n}
	=
	-\frac{315}{32}U_{7,n}\xi^3
	+
	\text{lower powers of $\xi$}.
\end{equation}

The complete derivatives of the normalized summands are expressed through
the complete exponential Bell polynomials:
\begin{equation}
	\frac{h_{n,p}^{(m)}(0)}{h_{n,p}(0)}
	=
	Y_m\!\left(
	\lambda_{1,n},\ldots,\lambda_{m,n}
	\right).
	\label{eq:bell-polynomial}
\end{equation}
Since the $\cC^2$-eigenvalue is $i$, Proposition~\ref{prop:eigenpath}
forces the summed Taylor coefficients of every degree
$m\not\equiv0\pmod4$ to vanish.  The resulting expressions are
polynomials in $\xi$, so each projective coefficient vanishes separately.
Extracting the highest projective coefficient at orders $3,5,6,$ and $7$
gives the following hierarchy.

\begin{proposition}[Projective Coxeter hierarchy]
	\label{prop:hierarchy}
	With $q_n$, $A_n$, and $X_n$ defined in
	\eqref{eq:bbbg-qn}, \eqref{eq:def-An}, and \eqref{eq:def-Xn},
	respectively, and with $U_{3,n},U_{4,n},U_{5,n},U_{7,n}$ as above,
	the symmetry-forced Taylor vanishings give
	\begin{equation}
		\sum_{n=0}^\infty q_nU_{3,n}=0,
	\end{equation}
	\begin{equation}
		\sum_{n=0}^\infty
		q_n\bigl(A_nU_{4,n}-U_{5,n}\bigr)=0,
	\end{equation}
	\begin{equation}
		\sum_{n=0}^\infty
		q_n\bigl(
		X_nU_{4,n}
		-2A_nU_{5,n}
		+U_{3,n}^2
		\bigr)=0,
	\end{equation}
	and
	\begin{equation}
		\sum_{n=0}^\infty
		q_n\bigl(
		U_{3,n}U_{4,n}-U_{7,n}
		\bigr)=0.
	\end{equation}
\end{proposition}

Expanding these compact relations gives the following explicit
consequences.

\begin{corollary}[Fifth-order identity]
	One has
	\begin{equation}
		\sum_{n=0}^\infty
		(4n+1)\frac{\binom{2n}{n}^6}{2^{12n}}
		\left[
		A_n\left(8H_{2n}^{(4)}-H_n^{(4)}\right)
		-16H_{2n}^{(5)}+H_n^{(5)}
		\right]
		=
		\frac{8\zeta(5)}{\pi^4}s_{4,0}.
	\end{equation}
\end{corollary}

\begin{corollary}[Sixth-order identity]
	One has
	\begin{equation}
		\sum_{n=0}^\infty
		(4n+1)\frac{\binom{2n}{n}^6}{2^{12n}}
		\left[
		X_n\left(8H_{2n}^{(4)}-H_n^{(4)}\right)
		-2A_n\left(16H_{2n}^{(5)}-H_n^{(5)}\right)
		+32\bigl(H_{2n}^{(3)}\bigr)^2
		\right]
		=
		\frac{8\zeta(3)^2}{\pi^4}s_{4,0}.
	\end{equation}
\end{corollary}

\begin{corollary}[Seventh-order identity]
	One has
	\begin{equation}
		\sum_{n=0}^\infty
		(4n+1)\frac{\binom{2n}{n}^6}{2^{12n}}
		\left[
		64H_{2n}^{(7)}
		-
		\bigl(8H_{2n}^{(3)}-\zeta(3)\bigr)
		\bigl(8H_{2n}^{(4)}-H_n^{(4)}\bigr)
		\right]
		=
		\frac{8\zeta(7)}{\pi^4}s_{4,0}.
	\end{equation}
\end{corollary}
\section{Concluding structural remarks}
\label{sec:conclusion}

The main theme of this paper is that several hypergeometric identities
arising from elliptic and Bessel moments are governed by the same
$W(D_5)$ symmetry once they are placed inside Mishev's completed
$L$-function.  What appear at first as rather different evaluations
then become different manifestations of a common parameter geometry.

For the complementary elliptic moments, the three families
$K'^2$, $E'K'$, and $E'^2$ fit naturally into a single contiguous
configuration.  Their hypergeometric representations are related not
only by formal parameter shifts, but also by the reflection and collision
structure of the completed $L$-function.  The Barnes argument provides
a complementary viewpoint: differential relations among elliptic
integrals are transported into contiguous relations in hypergeometric
parameter space.  Thus the analytic structure of the elliptic moments
and the algebraic structure of the hypergeometric transformations are
closely aligned.

The BBBG four-Bessel moment reveals another aspect of the same picture.
Here a harmonic factor that initially appears external to the
hypergeometric series is produced by differentiation in parameter
space.  A $W(D_5)$ reflection then moves the resulting deformation to
a simpler locus, while the symmetric point $z_0$ links the global
transformation theory with the local action of the Weyl group.  In this
sense, the Bessel calculation serves as a bridge between the global
symmetry of the completed $L$-function and the local symmetry of its
parameter space.

This local viewpoint leads to the eigenpath principle.  At a fixed
point of a finite-order Weyl element, the linearized group action
selects distinguished tangent directions, and invariance becomes a
selection rule for Taylor coefficients.  Parameter differentiation then
converts these symmetry constraints into harmonic identities.  The
higher-order calculations show that this is not merely a first-order
phenomenon: the projective geometry of the relevant eigenspaces carries
additional arithmetic information, reflected in the appearance of odd
zeta values and products of zeta values.

The resulting picture therefore has two complementary levels.  Globally,
$W(D_5)$ organizes reflections, collisions, cancellations, and contiguous
relations.  Locally, its tangent representation organizes parameter
derivatives and Taylor-selection identities.  Mishev's completed
$L$-function provides the common setting in which these two levels meet.
From this perspective, the elliptic-moment formulas, the BBBG
hypergeometric identity, and the harmonic special-value relations are
best viewed not as isolated evaluations, but as different expressions
of a single symmetry-driven hypergeometric structure.
\clearpage
\section*{Acknowledgements}

\paragraph{Use of artificial intelligence.}
The research reported in this article, including computational
exploration, proof strategy, the checking of calculations, and the
preparation and proofreading of the manuscript, was carried out by the
author with the assistance of the AI system ChatGPT (OpenAI).
All mathematical statements, computations, and proofs have been
checked and verified by the author, who takes full responsibility for
the correctness and integrity of the article.

\nocite{*}
\bibliographystyle{amsplain}
\bibliography{references}

\end{document}